\documentclass[preprint,11pt]{elsarticle}

\usepackage{lipsum}
\makeatletter
\def\ps@pprintTitle{
 \let\@oddhead\@empty
 \let\@evenhead\@empty
 \def\@oddfoot{}
 \let\@evenfoot\@oddfoot}
\makeatother

\usepackage[letterpaper,left=1in,right=1in,top=1in,bottom=1in]{geometry}
\usepackage{setspace}

\usepackage[utf8]{inputenc} 
\usepackage[T1]{fontenc}    
\usepackage{hyperref}       
\usepackage{url}            
\usepackage{booktabs}       
\usepackage{amsfonts}       
\usepackage{nicefrac}       
\usepackage{microtype}      
\usepackage{lipsum}
\usepackage{graphicx}
\graphicspath{ {./images/} }
\usepackage{soul}
\usepackage{ wrapfig}
\usepackage{subcaption}
\usepackage{float}
\usepackage{amsfonts,amsthm}
\usepackage{amssymb,amsmath,latexsym,tensor}
\usepackage{mathrsfs}
\usepackage{enumerate}
\usepackage{savesym}
\usepackage{graphicx}
\usepackage{adjustbox}
\usepackage[pdftex]{color}
\savesymbol{AND}
\usepackage{epstopdf}
\usepackage{algorithm}
\usepackage{algorithmic}

\DeclareMathOperator*{\argmin}{arg\,min}
\DeclareMathOperator*{\argmax}{arg\,max}
\usepackage{mathtools}
\DeclarePairedDelimiterX{\norm}[1]{\lVert}{\rVert}{#1}
\DeclarePairedDelimiterX{\inp}[2]{\langle}{\rangle}{#1, #2}
\newtheorem{theorem}{Theorem}[section]

\newtheorem{assumption}[theorem]{Assumption}
\newtheorem{definition}[theorem]{Definition}
\newtheorem{proposition}[theorem]{Proposition}
\newcommand{\indep}{\perp \!\!\! \perp}
\usepackage{cleveref}
\usepackage{mathtools}
\DeclarePairedDelimiter{\ceil}{\lceil}{\rceil}
\usepackage{array,multirow}

\begin{document}

\begin{frontmatter}

\title{A Computational Framework for Solving Nonlinear Binary Optimization Problems in Robust Causal Inference}



\author[mymainaddress]{Md Saiful Islam}
\ead{islam.m@northeastern.edu}

\author[mymainaddress]{Md Sarowar Morshed}
\ead{morshed.m@northeastern.edu}

\author[mymainaddress]{Md. Noor-E-Alam\corref{mycorrespondingauthor}}
\cortext[mycorrespondingauthor]{Corresponding author}
\ead{md.alam@northeastern.edu}

\address[mymainaddress]{Department of Mechanical and Industrial Engineering, Northeastern University, Boston, MA 02115, USA}

\begin{abstract}
Identifying cause-effect relations among variables is a key step in the decision-making process. While causal inference requires randomized experiments, researchers and policymakers are increasingly using observational studies to test causal hypotheses due to the wide availability of data and the infeasibility of experiments. The matching method is the most used technique to make causal inference from observational data. However, the pair assignment process in one-to-one matching creates uncertainty in the inference because of different choices made by the experimenter. Recently, discrete optimization models are proposed to tackle such uncertainty\textcolor{black}{;} however, they produce 0-1 nonlinear problems and lack scalability. In this work, we investigate this emerging data science problem and develop a unique computational framework to solve the robust causal inference test instances from observational data with continuous outcomes. In the proposed framework,
we first reformulate the nonlinear binary optimization problems as feasibility problems. By leveraging the structure of the feasibility formulation, we develop greedy schemes that are efficient in solving robust test problems. In many cases, the proposed algorithms achieve a globally optimal solution. We perform experiments
on real-world datasets to demonstrate the effectiveness of the proposed algorithms and compare our results with the state-of-the-art solver. Our experiments show that the proposed algorithms significantly outperform the exact method in terms of computation time while achieving the same conclusion for causal tests.
Both numerical experiments and complexity analysis demonstrate that the proposed algorithms ensure the scalability required for harnessing the power of big data in the decision-making process. Finally, the proposed framework not only facilitate\textcolor{black}{s} robust decision making through big-data causal inference, but it can also
be utilized in developing efficient algorithms for other nonlinear optimization problems such as quadratic
assignment problems. 
\end{abstract}

\begin{keyword}
Analytics \sep Discrete Optimization \sep Causal Inference \sep Big Data \sep Observational study 
\end{keyword}

\end{frontmatter}


\section{Introduction}
In this paper, we consider an emerging data science problem in robust causal inference and develop a unique computational framework consisting of \textcolor{black}{a} novel reformulation technique and innovative algorithms to facilitate decision making from large-scale observational data. As a natural process of digitization, we are continuously generating data on our health, behavior, mood, choices, and physical activity which creates a fertile field of prospective (collecting data that is generated as a natural process over time) or retrospective (experimenting on already collected data) observational experiments. Such experiments on non-randomized data are being used in identifying cause-effect relationships among variables to make informed policy decisions in public and private sectors \citep{nikolaev2013balance}. Policy decisions across different domains are made by intervening in different socio-economic variables or process parameters (treatment), and measuring their causal effect on the desired outcome. Even though randomized experiments are the gold standard for cause-effect analysis, it is often infeasible due to legal or ethical reasons. In addition, controlled experiments can be expensive and inapplicable to events that have already occurred. For instance, we might be interested in the effect of cloudy weather on bike rentals or the effect of eating fast-food on children's learning ability. Hence, in many decision-making processes, we are constrained to use observational data only, and rapid digitization is making this more prevalent.

Identifying \textcolor{black}{a} causal effect or testing \textcolor{black}{a} causal hypothesis from observational data is prone to confounding bias because of distributional differences between treated and control samples on the measured covariates \citep{stuart2010matching}. A common strategy adopted by researchers across different domains, known as the matching method, is to adjust for observed covariates to reduce the confounding bias. The matching method aims to restore the properties of a randomized experiment by finding a control group that is identical to the treatment group in terms of the joint distribution of the measured covariates \citep{stuart2010matching,nikolaev2013balance,sauppe2014complexity,sauppe2017role}. Due to the well-established methodological framework, matching methods have been used in many disciplines\textcolor{black}{,} including public health \citep{islam2019robust}, economics \citep{dehejia1999causal}, sociology \citep{gangl2010causal}, and education \citep{zubizarreta2014matching}. However, forming matched pairs by assigning treated samples to control samples is a data mining problem \citep{morucci2018hypothesis}. The matched pair construction is done by minimizing a single criterion between the treated and control samples such as multivariate distance \citep{rubin:1979}, the probability of receiving treatment \citep{rosenbaum1983central}, or the chi-squared test statistic \citep{king:2011,nikolaev2013balance}. Even so, finding pairs that balance the empirical distributions of the study subgroups is a difficult problem. In the past, this problem was solved with network flow algorithms \textcolor{black}{that} pursued covariate balance indirectly and relied heavily on iterative post-assignment balance checking \citep{king:2011}. These methods involve a significant amount of guesswork and the experimenter has little to no control over the matching process \citep{hill2011bayesian}. In recent years, discrete optimization models have been developed \citep{nikolaev2013balance,zubizarreta2012using,zubizarreta2014matching,zubizarreta2015stable} to solve matching problems that directly aim to minimize the imbalance metrics. Use of Mixed Integer Programming (MIP) models put the experimenter in control of the matching process and provides the flexibility of including higher-order moments and multivariate moments as the imbalance measure.

Traditional matching techniques, including the recent methods using MIP models, choose a single set of matched pairs while other possible sets of matched pairs may exist with equal or similar match quality. Ignoring other equally or almost equally good sets of matched pairs creates a source of uncertainty in the inference which is dependent on the choice of the experimenter over different matching methods \citep{morucci2018hypothesis}. For instance, in an attempt to evaluate the effect of the hospital readmission reduction program (HRRP) \citep{mcilvennan2015hospital} on non-index readmission (readmission to a hospital that is different from the hospital that discharged the patient), \citet{islam2019robust} showed that with more than 15,000 matched pairs, one experimenter can find HRRP as a cause of \textcolor{black}{higher} non-index readmission while another experimenter can find the opposite. Such uncertainty becomes more prominent for observational studies involving big data, as the chance of having multiple sets of good matches increases with the increase of the size of the data. \citet{hacking:2018} refers \textcolor{black}{to} such uncertainty in statistical inference as the \textit{Hacking Interval} in the context of linear regression and popular machine learning algorithms like K-Nearest Neighbour (KNN) and Support Vector Machine (SVM). \textcolor{black}{There are a variety of sensitivity analysis techniques available in causal inference literature; however, they mostly focus on the sensitivity of experiment design (i.e., confounding effect of unmeasured variables) and assumptions. Some of those techniques identify the bound of allowable unmeasured confounding (see Ch. 4 in \citep{rosenbaum2002observational}) when others consider the heterogeneity of the effect \citep{fogarty2020studentized}, or aim to develop non-parametric techniques \citep{howard2021uniform}. Nonetheless, these sensitivity analysis methods do not consider the uncertainty in the post-study design phase due to the choices made by an experimenter.}

To test \textcolor{black}{a} causal hypothesis that is robust to the experimenter's choice of the matching algorithm, \citet{morucci2018hypothesis} proposed discrete optimization-based tests for causal hypothesis with binary and continuous outcomes. The robust causal hypothesis tests explore all possible pair assignments by computing the maximum and minimum test statistics given a good set of matches.
While the robust tests proposed by \citet{morucci2018hypothesis} address the uncertainty in inference, the integer programming models produce nonlinear-binary optimization problems that are difficult to solve even for small instances. 

In general, \textcolor{black}{the} nonlinear problem with binary variables is considered one of the challenging problems yet, such problems are abound in science and engineering applications \citep{anthony2017quadratic,murray2010algorithm}. \textcolor{black}{Specifically}, optimization problems arising in computer vision, machine learning and statistics are often nonlinear in nature when posed as MIP. On the other hand, \citet{bertsimas2016best} show that if the problem is computationally tractable, \textcolor{black}{the} MIP approach can produce significantly better solution\textcolor{black}{s} than the numerical optimization techniques commonly used in data science. Unfortunately, there are very few successful solution algorithms available in the literature and even the linearly constrained problem with quadratic objective, which is considered the simplest case among nonlinear problems, is NP-hard \citep{murray2010algorithm}. In optimization literature, there are four major approaches to solve general nonlinear integer programming problems: continuous reformulation, quadratic reformulation, linearization with piecewise linear functions, and \textcolor{black}{a} heuristic/metaheuristic approach. The continuous reformulation aims to transform the nonlinear binary problem into an equivalent problem of finding global optimum in continuous variables through the use of algebraic, geometric and analytic techniques \citep{murray2010algorithm} or relaxing the binary constraint and adding a penalty in the objective \citep{lucidi2010exact}. However, such reformulation techniques often introduce exponentially large number of variables to the model and can make it even more difficult to solve \citep{murray2010algorithm}.

Similar to the continuous reformulation, quadratization adds a large number of variables in constructing an equivalent problem. For instance, quadratic reformulation of a nonlinear problem with $n$ binary variables will need at least $2^{n/2}$ auxilary variables (see Theorem 1 in \citet{anthony2017quadratic}). Classical linearization techniques also introduce \textcolor{black}{an} exponentially large number of additional variables \citep{anthony2017quadratic} and will be impractical for large-scale problems arising from big data. However, several heuristic algorithms have been successful in solving moderate\textcolor{black}{ly} sized problems with binary variables when the objective function is quadratic \citep{boros2007local,glover2002one}. Researchers often use metaheuristics like \textcolor{black}{the} genetic algorithm \citep{gopalakrishnan2015operational} and neighbourhood search algorithm \citep{claudia127online} to solve nonlinear binary optimization problems. In recent years, a variety of commercial nonlinear programming solvers have been developed\textcolor{black}{;} however, their applicability is limited to very small scale instances (see \citet{cafieri2017mixed}). A more related work is \citet{islam2019robust} which developed scalable and efficient algorithms to solve nonlinear binary problems in robust causal hypothesis tests with binary outcome data. However, the proposed solution approach is very specific to the objective function structure and not generalizable to \textcolor{black}{the problem that deals with continuous outcomes}.  

\textcolor{black}{On the other hand, modern MIP solvers (i.e., Gurobi, Cplex, Mosek) have gained significant computational power over the last few decades due to the algorithmic development and improvement of hardware capability. For instance, we can solve many MIPs in seconds today which, 25 years ago, would have taken 71,000 years to solve \citep{bertsimas2019machine}. Usually, MIP solvers use a combination of branch-and-bound, cutting plane, and group-theoretic approaches to solve practical problems \citep{bixby2012brief}. Instead of employing specific branching rules, MIP solvers often use a hybrid branching strategy by combining rules like \textit{non-chimerical branching} \citep{fischetti2012branching}, \textit{reliability branching}, and \textit{inference branching} \citep{martin2005branching}. Cutting planes are a pivotal tool for the solvers, especially when solving integer linear programs (ILPs). Quadratic integer programs (QIP) also take advantage of the branch-and-bound algorithm where at each node, a quadratic program is solved using efficient techniques like barrier methods. In addition to these techniques, MIP solvers often use efficient preprocessing methods to tighten the original formulation and reduce the size of the problem \citep{Achterberg2013}. There are a wide range of starting and improvement heuristics available in the solvers that provide a good incumbent solution or a sufficient one when the problem is intractable (see the survey by \citet{fischetti2010heuristics}).}

\textcolor{black}{Apart from the algorithmic developments, modern MIP solvers take advantage of hardware technology by running multiple optimizers concurrently on multiple threads and choosing the best one, as we often do not know which algorithm, branching rule, or their combination would be most efficient for the problem at hand. To further increase the efficiency of MIP solver, researchers are exploring different machine learning \citep{he2014learning,alvarez2014supervised} and reinforcement learning \citep{etheve2020reinforcement} techniques for better branching and predicting the subtree size at a node. However, these MIP solvers cannot solve a general nonlinear problem. To leverage the improvements of the MIP solvers, we need a linear or at least a quadratic formulation of the current nonlinear problems in robust causal inference.     
}

In this paper, we consider the nonlinear-binary optimization models of robust causal inference tests proposed by \citet{morucci2018hypothesis} for continuous outcome data and develop a computational framework that can handle large-scale observational data. First, we propose a unique approach to reformulate the nonlinear binary optimization problems as equivalent and less-restrictive feasibility problems. By exploiting the structure of the feasibility problem, we then develop greedy algorithms to solve the original robust causal hypothesis test problems. The reformulation into \textcolor{black}{a} feasibility problem also provides an opportunity to formulate the robust test problems as quadratic integer programming problems which can take advantage of recent development\textcolor{black}{s} in commercial MIP solvers. Without this reformulation, we cannot use MIP solvers to get guaranteed optimality. A major advantage of the reformulation approach is that unlike state-of-the-art quadratization techniques, here, we do not need to add any additional variables. Moreover, the proposed unique reformulation approach can be leveraged to model general nonlinear and quadratic optimization problems (i.e., Quadratic Assignment Problem) as feasibility problems and to develop efficient algorithms. Nonetheless, the proposed algorithms are very efficient in solving practical sized problems and scalable to very large-scale problems. Our numerical experiments on real-world datasets show that the proposed computational framework provides the same inference to robust tests while taking a fraction of the time required by the exact method. The time complexity demonstrates that the developed algorithms are scalable enough to harness the power of big data in performing robust causal analysis which consequently will provide robust policy decisions.

The remainder of the paper is organized as follows. We discuss the matching method and the robust causal hypothesis test with continuous outcomes (i.e., Robust Z-test) along with its challenges in section 2. We provide a feasibility formulation of the robust Z-test in section 3. In section 4, we develop greedy algorithms to solve the reformulated Z-test and analyze the properties and complexity of the proposed algorithms. We apply our algorithms to three real-world datasets of varying sizes in section 5 and compare our result with Gurobi \textcolor{black}{and ILP-based heuristic proposed by \citet{morucci2018hypothesis}}. Finally, we provide concluding remarks in section 6.

\section{Matching method and Robust test}
In this section, we discuss the matching method in general and the robust causal hypothesis test with continuous outcomes. We first discuss the matching method, introduce necessary notations and required assumptions to make causal inference from observational data. We then discuss the integer programming model for \textcolor{black}{a} robust causal hypothesis test with continuous outcomes, existing methods to solve the robust test, and issues with the current solution method.

\subsection{Matching Method}
In this paper, we consider the potential outcome framework \citep{holland1986stat} which has led to the development of matching methods. Under the potential outcome framework, treatment effect or causal effect is measured by comparing the counterfactual: difference in the potential outcomes of a sample unit in both treated and control scenario. The fundamental problem in causal inference is that we can only observe one scenario (either treated or control) for each sample. A sample $i \in \mathscr{S}$ where $\mathscr{S}:= \{ 1,2,\cdots, N \}$ can only have outcome $Y_i = Y_i^1 T_i + Y_i^0 (1-T_i)$ where $T_i \in \{0,1\}$ \citep{holland1986stat}. Here $Y_i^1$ represents outcome under treatment (i.e., $T_i = 1$) and $Y_i^0$ represents outcome without treatment (i.e., $T_i = 0$). This problem is overcome by considering the average of treatment effect $\tau = Y^1 - Y^0$ over study population $\mathscr{S}$. While this strategy works for randomized experiments, experiments with observational data produce biased inference due to the distributional differences of the groups which received treatment (i.e., $\mathscr{T} \subset \mathscr{S}$) and the group which did not (i.e., $\mathscr{C} \subset \mathscr{S}$) on some pre-treatment covariates.

The matching method provides an unbiased estimation of \textcolor{black}{the} treatment effect by identifying pairs $(t,c)$ where $t \in \mathscr{T}$ and $c \in \mathscr{C}$ or subsets $\mathcal{T} \subset \mathscr{T}$ and $\mathcal{C} \subset \mathscr{C}$ are matched exactly in terms of their covariate set $\mathbf{X} \in \mathcal{X}$ \citep{rosenbaum1983central,stuart2010matching}. However, it is almost impossible to find exact matches even with a small number of covariates \citep{zubizarreta2012using,rosenbaum1983central}. A large number of matching methods have been developed to make $(t,c)$ pairs or $(\mathcal{T}, \mathcal{C})$ subsets as similar as possible in terms of $\mathbf{X}$ \citep{stuart2010matching,nikolaev2013balance}. In this paper, we will restrict the scope to one-to-one matching and consider the matching algorithms that identify $(t,c)$ pairs. Some commonly used matching methods are Propensity Score Matching \citep{rosenbaum1983central}, Nearest Neighbor Matching \citep{stuart2010matching}, Optimal Matching \citep{rosenbaum1989optimal}, Mahalanobis Distance Matching \citep{rubin:1979}, and Genetic Matching \citep{diamond2013genetic}.  One of the popular methods (if not the most popular) \citep{stuart2010matching,zubizarreta2012using} is the propensity score matching method \citep{rosenbaum1983central} that employs a logistic model to estimate each sample's propensity of receiving treatment and find the $(t,c)$ pairs \textcolor{black}{that are} minimizing the differences in their propensity scores. The matching process is repeated and evaluated iteratively until the desired quality of the matches is achieved. Once a suitable set of matches are identified, we can test the null hypothesis \eqref{hp:ate} and \eqref{hp:att} to make causal inference. $\mathcal{H}_0^{SATE}$ in \eqref{hp:ate} tests the zero sample average treatment effect (SATE) hypothesis on the whole sample set and $\mathcal{H}_0^{SATT}$ in \eqref{hp:att} tests the zero sample average treatment effect hypothesis on the treated (SATT) samples. Here, we make the assumptions that are commonly used in causal inference literature (see \ref{app:1}). 
\begin{align}
    & \mathcal{H}_0^{SATE} := \mathbb{E}_{Y|X} [ Y^1-Y^0 \vert \mathbf{X} ] = 0  \label{hp:ate}\\
    & \mathcal{H}_0^{SATT} := \mathbb{E}_{Y|X}[ Y^1-Y^0 \vert \mathbf{X}, T=1 ]= 0 \label{hp:att}
\end{align}

\subsection{Uncertainty due to the choice of the experimenter}
As we mentioned in section 2.1, matching methods aim to minimize a specific set of criteria to find the matched pairs $(t,c)$ in one-to-one matching. In causal inference literature, there are plentiful algorithms to find such matched pair sets\textcolor{black}{;} however, \textcolor{black}{it} offer\textcolor{black}{s} little to no clear guidance on the choices of matching procedure \citep{morgan2015counterfactuals}. A common practice among researchers and practitioners is to use widely-adopted and -cited techniques and software \citep{morucci2018hypothesis}. Therefore, starting with the same dataset and the same matching objectives, two experimenters can get two different sets of matched pairs. As the pair assignment process does not consider outcomes, traditional inference techniques do not guarantee the same inference by both experimenters. For instance, treated unit $t\in \mathscr{T}$ can have multiple potential assignments $\left \{ c_1, c_2, \cdots, c_n \right  \} \in \mathscr{C}$ with equal match quality but different outcomes. Each possible assignment will potentially have a different treatment effect: $Y_t^1 - Y_{c_1}^0 \neq  Y_t^1 - Y_{c_2}^0 \neq \cdots \neq Y_t^1 - Y_{c_n}^0$ which creates an uncertainty in the inference. Moreover, this phenomenon is more prominent in big data observational studies due to the availability of more pairing options.

\subsection{Robust approach for causal hypothesis tests }
To make a causal inference that is robust to the choice of the experimenter, recently, \citet{morucci2018hypothesis} proposed a new methodology based on discrete optimization technique. Before discussing its difference with the classical method of making causal inference \citep{rosenbaum1985constructing,holland1986stat}, we need to define the pair assignment variables $a_{i,j}$ and the set of good matches $\mathscr{M}$.

\begin{definition}
\label{def:good_match}
\textup{(A set of Good Matches)} $\mathscr{M}$ includes treated samples $\mathscr{T} \subset \mathscr{S}$ and control samples $\mathscr{C} \subset \mathscr{S}$ that satisfies user-defined matching criteria such that $\mathscr{M} := \left \{(t_i,c_j) \in (\mathscr{T} \times \mathscr{C}) : t_i \simeq c_j \vert \mathbf{X}    \right \}$.
\end{definition}
The set of good matches $\mathscr{M}$ can be stored in a logical matrix $D^{|\mathscr{T}| \times |\mathscr{C}|}$. Here, an element in $D$, $d_{i,j} = 1$ if treated unit $i$ is a good match to a control unit $j$, otherwise 0.

\begin{definition}
\label{def:pair_op}
\textup{(Pair Assignment Operator)} $a_{ij} \in \left \{ 0,1 \right \}$ is the pair assignment operator where $a_{ij} = 1$ if treated unit $t_i \in \mathscr{T}$ is paired with a control unit $c_j \in \mathscr{C}$ and the pair $(t_i,c_j) \in \mathscr{M}$; $a_{ij} = 0$ otherwise. 
\end{definition}

Here, we are interested \textcolor{black}{in testing} the causal hypothesis \eqref{hp:ate} or \eqref{hp:att} given a good set of matches $\mathscr{M}$. The classical approach selects one matched pair set from $\mathscr{M}$, calculates the test statistic $\Lambda$, and makes inference based on $\Lambda$ or the corresponding $P$-value. On the other hand, the robust approach explores all possible pair assignments within $\mathscr{M}$ without replacement by calculating the maximum and minimum test statistics $(\Lambda_{max}, \Lambda_{min})$. Using $(\Lambda_{max}$, $\Lambda_{min})$, a robust test can be defined as the following.  

\begin{definition}
\label{defrobust_test}
\textup{(Robust Test)} Given a level of significance $\alpha$ to test the hypothesis $\mathcal{H}_0$ and $(\Lambda_{max}, \Lambda_{min})$ calculated from $\mathscr{M}$, the test is called $\alpha$-robust \textcolor{black}{if \textup{$\lvert$p-value($\Lambda_{max}$) - p-value($\Lambda_{min}$)}$\rvert\leq \alpha$.} The test $\mathcal{H}_0$ is called absolute-robust when \textup{p-value($\Lambda_{min}$) $=$ p-value($\Lambda_{max}$)}.
\end{definition}

However, computing the test statistics require to solve binary optimization problems and their structure depends on the nature of the test statistics. For example, the McNemar\textcolor{black}{'s} test \citep{mcnemar1947note} and the Z-test \citep{low2016enhancing} proposed by \citet{morucci2018hypothesis} for binary and continuous data, respectively, produce nonlinear binary optimization problems. While they ensure robustness in inference, nonlinear-binary optimization problems are extremely difficult to solve, even for smaller instances. \citet{islam2019robust} developed efficient algorithms for the McNemar\textcolor{black}{'s} test with binary outcomes by converting the original nonlinear problem into a counting problem. The robust tests with the continuous outcome, on the other hand, remain difficult to solve for practical sized problems. In the following section, we discuss the robust Z-test model, the current solution approach and challenges with the current approach.

\subsection{Robust Z-test and Challenges}
To test the zero causal effect hypothesis with continuous outcomes, one can consider the canonical Z-test \citep{morucci2018hypothesis} with the test statistics $\Lambda :=Z(\mathbf{a}) = \frac{\sqrt{n}(\hat{\tau} - 0)}{\hat{\sigma}}$. Here, $n$ is the number of matched pairs, $\hat{\tau} = \frac{1}{n} \sum_{i \in \mathscr{T}}\sum_{j \in \mathscr{C}} (y^t_i - y^c_j)a_{i,j}$ is the average treatment effect among the matched pairs, and $\hat{\sigma}$ is the sample standard deviation of the treatment effect. Given a set of good matches $\mathscr{M}$, an integer programming formulation of the Z-test proposed by \citet{morucci2018hypothesis} is provided in \eqref{z:obj}-\eqref{z:bin}.   
\begin{align}
     \textrm{max/min } Z(\mathbf{a})  & = \frac{\frac{1}{\sqrt{n}} \sum_{i \in \mathscr{T}}\sum_{j \in \mathscr{C}} (y^t_i - y^c_j)a_{i,j}}{\sqrt{\frac{1}{n} \sum_{i \in \mathscr{T}}\sum_{j \in \mathscr{C}} [(y^t_i - y^c_j)a_{i,j}]^2 - (\frac{1}{n} \sum_{i \in \mathscr{T}}\sum_{j \in \mathscr{C}} (y^t_i - y^c_j)a_{i,j})^2} } \label{z:obj}\\
  \textrm{subject to: }   & \sum_{i \in \mathscr{T}} a_{ij} \leq 1 \quad \forall j \label{z:angn1} \\
    & \sum_{j \in \mathscr{C}} a_{ij} \leq 1 \quad \forall i \label{z:angn2} \\
    & a_{i,j} \leq d_{i,j} \quad \forall i,j \label{z:dij} \\
    & \sum_{i \in \mathscr{T}} \sum_{j \in \mathscr{C}} a_{ij} = n \label{z:angn_n}\\
    & a_{ij} \in \left \{ 0,1 \right \} \quad \forall i,j \label{z:bin}
\end{align}
\textcolor{black}{Here, the objective function $Z(\mathbf{a})$ represents the test statistic $\Lambda$ and the denominator in the objective function (equation \eqref{z:obj}) is the sample standard deviation ($\hat{\sigma}$) of the treatment effect.} As we can see, to compute the test statistic, we have to solve above nonlinear-binary optimization problems. The current solution approach proposed in \cite{morucci2018hypothesis} linearizes the above model by imposing an upper bound on $\hat{\sigma}$: $\sum_{i \in \mathscr{T}}\sum_{j \in \mathscr{C}} [(y^t_i - y^c_j)a_{i,j}]^2 \leq b_l$ and replacing the objective with max/min $\sum_{i \in \mathscr{T}}\sum_{j \in \mathscr{C}} (y^t_i - y^c_j)a_{i,j}$. Therefore, the linearized Z-test model takes the following form.
\begin{align}
     \textrm{max/min } & \sum_{i \in \mathscr{T}}\sum_{j \in \mathscr{C}} (y^t_i - y^c_j)a_{i,j} \label{z:lin}\\
  \textrm{subject to: }   & \sum_{i \in \mathscr{T}}\sum_{j \in \mathscr{C}} [(y^t_i - y^c_j)a_{i,j}]^2 \leq b_l \quad \textrm{   and constraints \eqref{z:angn1}-\eqref{z:bin}} \label{z:bl}
\end{align}

The integer linear program (ILP)-based heuristic of \citet{morucci2018hypothesis} solve the model \eqref{z:lin}-\eqref{z:bl} iteratively by changing the upper bound of standard deviation $b_l$ on a grid. Starting with a coarse grid of $b_l$, the proposed algorithm solves a series of ILPs and creates a new, refined mesh at each iteration. The iterative process continues on the grid of $b_l$ until the desired level of tolerance on the upper bound $b_l$ is achieved. While this innovative approach provides a working solution to a complex problem, it faces several challenges in practice. The first challenge is \textcolor{black}{that} at each grid point we need to solve an ILP  and after refining the grid of $b_l$ we have to solve the ILP with updated constraints. As we have to iterate over the grid of $b_l$ many times, the ILP has to be solved hundreds of times if not thousands. One can solve small ILP instances with commercial MIP solvers efficiently\textcolor{black}{;} however, in today's big data world such smaller problems are highly unlikely. In addition, solving hundreds of ILPs add a significant computational burden. The second challenge is \textcolor{black}{that} the range of $b_l$ can be very large which significantly increases the number of ILPs we have to solve to calculate the test statistics. For instance, in a case study with Bikeshare data \citep{fanaee2014event}, the range of $b_l$ is 1.12 million to 26.12 million. 

To overcome these challenges, in this paper, we reformulate the nonlinear-binary formulation of the Z-test into a less-restricted feasibility problem. By leveraging the structure of the feasibility problem, we develop greedy algorithms that are very efficient and scalable to big data observational experiments. The feasibility formulation also allows us convert the Z-test problems into a quadratic integer programs to take advantage of recent developments in MIP solvers.

\section{Reformulation of Z-test}
\textcolor{black}{The optimization model discussed for the robust Z-test in \eqref{z:obj}-\eqref{z:bin} has a fractional objective in the form of $ Z(\mathbf{a}) = \frac{f(x)}{g(x)}$). Both $f(x)$ and $g(x)$ can be written as a function of treatment effect between treated sample $i$ and control sample $j$: $(y^t_i - y^c_j)a_{i,j}$. In addition, we can bound $Z(\mathbf{a})$ in the range of $Z(\mathbf{a}) \leq |4|$ as beyond this range, we will have approximately zero area under the standard normal distribution curve and optimizing further beyond this range will not make any difference in the inference. For instance, let's assume we find a sub-optimal solution for the maximization problem, a set of treated-control pair assignments for which $Z(\mathbf{a}) = 4.10$. In theory, we can find a global optimal solution better than the current solution; however, improving the solution quality further will not change the robust inference of the hypothesis test. Using this property of the optimization model, we reformulate the robust Z-test problem as a feasibility problem.} 

To reformulate the robust Z-test as feasibility problem\textcolor{black}{,} let's assume \textcolor{black}{that} for the minimization model $\exists$ $a_{i,j}$ such that $Z(\mathbf{a}) \leq \gamma$ where, $\gamma$ is a scalar parameter. \textcolor{black}{We can iterate over different values of $\gamma$ to find its optimal value.} Then we have the following:
\begin{align}
    & \frac{1}{\sqrt{n}} \sum_{i\in \mathscr{T}}\sum_{j \in \mathscr{C}} (y^t_i - y^c_j)a_{i,j} \leq \gamma \times \sqrt{\frac{1}{n} \sum_{i\in \mathscr{T}} \sum_{j \in \mathscr{C}} [(y^t_i - y^c_j) a_{i,j}]^2 - (\frac{1}{n} \sum_{i\in \mathscr{T}}\sum_{j \in \mathscr{C}} (y^t_i - y^c_j)a_{i,j})^2 } \label{feas:root} \\
    & \frac{1}{n} (\sum_{i\in \mathscr{T}}\sum_{j \in \mathscr{C}} (y^t_i - y^c_j)a_{i,j})^2 \leq \gamma^2 \frac{1}{n} \sum_{i\in \mathscr{T}} \sum_{j \in \mathscr{C}} [(y^t_i - y^c_j) a_{i,j}]^2 - \gamma^2 (\frac{1}{n} \sum_{i\in \mathscr{T}}\sum_{j \in \mathscr{C}} (y^t_i - y^c_j)a_{i,j})^2 \label{feas}
\end{align}
It is important to note that in the above equation \eqref{feas}, we are taking the square of an inequality
\textcolor{black}{for further simplification. Both sides of the inequality \eqref{feas:root}, can be negative and non-negative depending on the intervals of sum of treatment effects ($\sum_{i\in \mathscr{T}}\sum_{j \in \mathscr{C}} (y^t_i - y^c_j)a_{i,j}$) and $\gamma$. Therefore, to ensure the correct direction of the inequality, we must consider four possible combinations of the sum of treatment effects and $\gamma$. We consider these combinations in four cases and discuss them in the following subsection.} \textcolor{black}{Before we discuss the four cases, let us introduce the proposed feasibility formulation by assuming} \textcolor{black}{equation} \eqref{feas} is a valid inequality. Simplifying \textcolor{black}{equation} \eqref{feas} further, we get the following quadratic constraint.
\begin{align}
    (1+\gamma^2 \frac{1}{n}) (\sum_{i\in \mathscr{T}}\sum_{j \in \mathscr{C}} (y^t_i - y^c_j)a_{i,j})^2 - \gamma^2 \sum_{i\in \mathscr{T}} \sum_{j \in \mathscr{C}} [(y^t_i - y^c_j) a_{i,j}]^2 \leq 0
\end{align}
Therefore, the nonlinear-binary optimization problem can be framed as the following binary-feasibility problem: $\exists \textrm{ a set of assignments } a_{i,j} \textrm{ so that } Z(\mathbf{a}) \leq \gamma$ with constraints \eqref{feas:quad}-\eqref{feas:bin}.
\begin{align}
    & (1+\gamma^2 \frac{1}{n}) (\sum_{i\in \mathscr{T}}\sum_{j \in \mathscr{C}} (y^t_i - y^c_j)a_{i,j})^2 - \gamma^2 \sum_{i\in \mathscr{T}} \sum_{j \in \mathscr{C}} [(y^t_i - y^c_j) a_{i,j}]^2 \leq 0 \label{feas:quad}\\
   & \sum_{i \in \mathscr{T}} a_{ij} \leq 1 \quad \forall j \\
   &\sum_{j \in \mathscr{C}} a_{ij} \leq 1 \quad \forall i 
\end{align} 
\begin{align}
   & a_{i,j} \leq d_{i,j} \quad \forall i,j \\
   & \sum_{i \in \mathscr{T}} \sum_{j \in \mathscr{C}} a_{ij} = n \label{feas:fixed_pair}\\
   & \textrm{Additional constraints to validate the inequality \eqref{feas:quad}} \\
   & a_{ij} \in \left \{ 0,1 \right \} \quad \forall i,j \label{feas:bin}
\end{align}

\paragraph{Advantage of the proposed feasibility formulation.} The reformulation \eqref{feas:quad}-\eqref{feas:bin} converts a nonlinear optimization problem into an equivalent and less-restrictive feasibility problem that offers several advantages. 

Firstly, it presents an opportunity to formulate the robust Z-test problem as a quadratic integer program (QIP) (see section 4). While general nonlinear integer optimization problems are difficult to solve, QIP formulation can benefit from the recent developments in MIP solvers. 

Secondly, feasibility formulation facilitates new algorithmic development. For example, we can show that the reformulated problem \eqref{feas:quad}-\eqref{feas:bin} is a convex-feasibility problem with binary variables. Solving the feasibility problem for any $\gamma$ and finding an optimal $\gamma$ in the range of $-4 \leq \gamma \leq 4$ with a binary search algorithm can solve the robust Z-test problem. Note that the range of $\gamma$ is fixed for all datasets and very small compared to $b_l$ used in the \textcolor{black}{ILP-based} approach, as beyond $Z(\mathbf{a}) = |4|$ will have approximately zero area under the standard normal distribution curve. \textcolor{black}{Similar to the Z-test, most parametric hypothesis tests (i.e., Student's \textit{t}-test, Welch's \textit{t}-test, \textit{F}-test, $\chi^2$-test) use test statistics in $\frac{f(x)}{g(x)}$ form where both $f(x)$ and $g(x)$ are functions of data, and the test statistics follow certain distributions. If those tests are conducted in the spirit of the \textit{hacking interval} proposed by \citet{hacking:2018}, regardless of the application domain we can use the bounded nature of the distribution of test statistics and its structure to create a feasibility formulation. The proposed feasibility reformulation scheme will provide guidance to develop scalable solution techniques for other robust hypothesis tests.} \textcolor{black}{With regard to} solving the feasibility problem, convex feasibility problems with continuous variables \textcolor{black}{have} been studied extensively in the optimization literature \citep{proj_book}, and iterative projection-based algorithms \citep{morshed2021sampling,convfeas1,convexfeas2,convexfeas3,convexfeas4} \textcolor{black}{have proven} proved to be very efficient in solving such problems. Unfortunately, to our best knowledge, there is no \textcolor{black}{efficient} algorithm available to solve convex feasibility problem with binary variables. Recently, projection\textcolor{black}{-}based algorithms are developed \citep{chubanov2015polynomial,chubanov2012strongly,basu2014chubanov} to solve linear feasibility problems with binary variables. We believe that new algorithms can be developed for solving convex feasibility problem with integer variables. Apart from causal hypothesis test problem, other types of optimization problems such as Quadratic Assignment Problem (QAP) \citep{pardalos1998quadratic} can be formulated as feasibility problem and can be solved with iterative projection-based algorithms. 

Finally, this unique reformulation simplifies the structure of the problem. By leveraging the structure, we can develop algorithms to solve such a computationally expensive problem. In that vein, we develop several greedy schemes to solve the robust Z-test problems that are efficient and scalable to harness the power of big data in causal analysis. 


\subsection{Cases for Minimization problem}
To simplify the test statistic $Z(\mathbf{a})$ and formulate it as a feasibility problem, we took the square of an inequality in equation \eqref{feas}. To ensure the validity of this inequality, we need to add additional constraints that lead to four possible cases. The resulting cases along with the constraints are presented below. It is important to note that $\hat{\sigma} = \sqrt{\frac{1}{n} \sum_{i\in \mathscr{T}} \sum_{j \in \mathscr{C}} [(y^t_i - y^c_j) a_{i,j}]^2 - (\frac{1}{n} \sum_{i\in \mathscr{T}}\sum_{j \in \mathscr{C}} (y^t_i - y^c_j)a_{i,j})^2 } > 0$, therefore, will not influence the cases.

\textbf{Case 1}: For this case, we consider $\gamma \geq 0$ and $\sum_{i\in \mathscr{T}}\sum_{j \in \mathscr{C}} (y^t_i - y^c_j)a_{i,j} \geq 0$. Now, the equation \eqref{feas:root} can be written as $\frac{1}{\sqrt{n}}\sum_{i\in \mathscr{T}}\sum_{j \in \mathscr{C}} (y^t_i - y^c_j)a_{i,j} \leq \gamma \hat{\sigma}$. Here, $\gamma \geq 0, \hat{\sigma} >0$ and $\sum_{i\in \mathscr{T}}\sum_{j \in \mathscr{C}} (T_i -C_j)a_{i,j} \geq 0$. So, both sides of the inequality are positive. Taking the square will not change the sign of the inequality. Hence, we will have the following constraints.
\begin{align}
    & (1+\gamma^2 \frac{1}{n}) (\sum_{i\in \mathscr{T}}\sum_{j \in \mathscr{C}} (y^t_i - y^c_j)a_{i,j})^2 - \gamma^2 \sum_{i\in \mathscr{T}} \sum_{j \in \mathscr{C}} [(y^t_i - y^c_j) a_{i,j}]^2 \leq 0 \label{min_case1_nonlin} \\
    & \sum_{i\in \mathscr{T}}\sum_{j \in \mathscr{C}} (y^t_i - y^c_j)a_{i,j} \geq 0 \label{min_case1_lin}
\end{align}

\textbf{Case 2}: For this case, we consider $\gamma \leq 0$ and $\sum_{i\in \mathscr{T}}\sum_{j \in \mathscr{C}} (y^t_i - y^c_j)a_{i,j} \leq 0$. As both sides of the inequality \eqref{feas:root} are negative, the sign of the quadratic constraint will change and we will have the following constraints. 
\begin{align}
    & (1+\gamma^2 \frac{1}{n}) (\sum_{i\in \mathscr{T}}\sum_{j \in \mathscr{C}} (y^t_i - y^c_j)a_{i,j})^2 - \gamma^2 \sum_{i\in \mathscr{T}} \sum_{j \in \mathscr{C}} [(y^t_i - y^c_j) a_{i,j}]^2 \geq 0 \label{min_case2_nonlin}\\
    & \sum_{i\in \mathscr{T}}\sum_{j \in \mathscr{C}} (y^t_i - y^c_j)a_{i,j} \leq 0 \label{min_case2_lin}
\end{align}

\textbf{Case 3}: For this case, we consider $\gamma \geq 0$ and $\sum_{i\in \mathscr{T}}\sum_{j \in \mathscr{C}} (y^t_i - y^c_j)a_{i,j} \leq 0$. As $\sum_{i\in \mathscr{T}}\sum_{j \in \mathscr{C}} (y^t_i - y^c_j)a_{i,j} \leq 0$, left side of the inequality \eqref{feas:root} is non-positive but, $\gamma \textcolor{black}{\geq} 0$ and $\hat{\sigma} >0$ that makes the right side non-negative. Since, this is true for all $\gamma \in \mathbb{R}$, $\gamma \geq 0$, then, we have,
\begin{align}
    & \frac{1}{\sqrt{n}}\sum_{i\in \mathscr{T}}\sum_{j \in \mathscr{C}} (y^t_i - y^c_j)a_{i,j} \leq \min_{\gamma \in \mathbb{R}} (\gamma \hat{\sigma}) = 0
\end{align}
Therefore, we will have the following constraints. For this case, our non-linear feasibility problem becomes a linear feasibility problem. 
\begin{align}
    & \sum_{i\in \mathscr{T}}\sum_{j \in \mathscr{C}} (y^t_i - y^c_j)a_{i,j} \leq 0
\end{align}

\textbf{Case 4:} For this case, we consider $\gamma \leq 0$ and $\sum_{i\in \mathscr{T}}\sum_{j \in \mathscr{C}} (y^t_i - y^c_j)a_{i,j} \geq 0$. As the left side of the inequality \eqref{feas:root} is non-negative and the right side is non-positive, the above equation only holds at equality: $\sum_{i\in \mathscr{T}}\sum_{j \in \mathscr{C}} (y^t_i - y^c_j)a_{i,j} = 0$. However, this specific case is redundant as it is already considered in the other cases. All of the above-mentioned cases and resulting constraints are summarized in table \ref{tab:min_cases}.

\textcolor{black}{Each of the above case has two conditions, one on the sum of treatment effect ($\sum_{i\in \mathscr{T}}\sum_{j \in \mathscr{C}} (y^t_i - y^c_j)a_{i,j}$) and another on $\gamma$.
The condition of sum of treatment effect $\sum_{i\in \mathscr{T}}\sum_{j \in \mathscr{C}} (y^t_i - y^c_j)a_{i,j} \geq 0$ (or $\leq 0$) depends on the assignment variable; therefore, this condition is included in the optimization model as a constraint in equation (19). For the case 1 in the minimization problem, this constraint takes the form of equation (22).
$\gamma$ is not dependent on the assignment variable, so we do not include it in the optimization model. The condition on $\gamma$ is considered after solving the assignment problems as discussed in the following section. }

\begin{table}[htbp]
  \begin{center}

  \caption{Cases and resulting constraints for the minimization problem feasibility formulation of robust Z-test.}
  {\small
  \adjustbox{max width=\textwidth}{
    \begin{tabular}{|c|c|c|}
    \hline
       Case & Case Constraints   & Quadratic Constraint \\
    \hline
       1 & $\gamma \geq 0$, $\sum_{i\in \mathscr{T}}\sum_{j \in \mathscr{C}} (y^t_i - y^c_j)a_{i,j} \geq 0$  &  $(1+\gamma^2 \frac{1}{n}) (\sum_{i\in \mathscr{T}}\sum_{j \in \mathscr{C}} (y^t_i - y^c_j)a_{i,j})^2 - \gamma^2 \sum_{i\in \mathscr{T}} \sum_{j \in \mathscr{C}} [(y^t_i - y^c_j) a_{i,j}]^2 \leq 0$ \\
    
    \hline
       2 & $\gamma \leq 0$, $\sum_{i\in \mathscr{T}}\sum_{j \in \mathscr{C}} (y^t_i - y^c_j)a_{i,j} \leq 0$   & $(1+\gamma^2 \frac{1}{n}) (\sum_{i\in \mathscr{T}}\sum_{j \in \mathscr{C}} (y^t_i - y^c_j)a_{i,j})^2 - \gamma^2 \sum_{i\in \mathscr{T}} \sum_{j \in \mathscr{C}} [(y^t_i - y^c_j) a_{i,j}]^2 \geq 0$ \\
    \hline
       3 & $\gamma \geq 0$, $\sum_{i\in \mathscr{T}}\sum_{j \in \mathscr{C}} (y^t_i - y^c_j)a_{i,j} \leq 0$   & No quadratic constraint\\
    \hline
       4 & $\gamma \leq 0$, $\sum_{i\in \mathscr{T}}\sum_{j \in \mathscr{C}} (y^t_i - y^c_j)a_{i,j} \geq 0$   &  No quadratic constraint (Redundant, only true when both inequalities are zero) \\
    \hline
    \end{tabular}}%
  \label{tab:min_cases}%
    }
  \end{center}
\end{table}%

The maximization model follows the same argument. For completeness, we provide the corresponding cases of the maximization problem in \ref{app:2}.

\section{Algorithmic Approach}
In this section, we develop greedy algorithms to solve the derived cases of the robust Z-test by exploiting the structure of the feasibility problem formulated in the previous section. First, we consider case 1 of the minimization problem. In case 1, apart from the assignment constraints, we have to satisfy the constraints \eqref{min_case1_nonlin} and \eqref{min_case1_lin} to calculate $\gamma^*$ in the range $\gamma \geq 0 $. Simplifying the constraint \eqref{min_case1_nonlin} will result in the following:
\begin{align}
&  ( \frac{n\gamma^2 }{n+\gamma^2}  )\sum_{i\in \mathscr{T}} \sum_{j \in \mathscr{C}} \left [ (y^t_i - y^c_j) a_{i,j} \right ]^2 \geq [ \sum_{i\in \mathscr{T}}\sum_{j \in \mathscr{C}} (y^t_i - y^c_j)a_{i,j} ]^2 \label{min_case1_nonlin_simple}
\end{align}
As our objective is to find a set of assignments $a_{i,j}$ that produces the smallest value of $\gamma$ while satisfying the constraints \eqref{min_case1_nonlin_simple} and \eqref{min_case1_lin}, we can exploit the structure of equation \eqref{min_case1_nonlin_simple}. Note that in equation \eqref{min_case1_nonlin_simple}, for \textcolor{black}{any} number of samples ($n$), the minimum possible value of $\gamma$ is possible when $\sum_{i\in \mathscr{T}} \sum_{j \in \mathscr{C}} [ (y^t_i - y^c_j) a_{i,j}  ]^2$ is maximized and $[ \sum_{i\in \mathscr{T}}\sum_{j \in \mathscr{C}} (y^t_i - y^c_j)a_{i,j} ]^2$ is minimized which can be written as the following optimization problem. 
\begin{align}
    & \max_{a_{i,j} \in \mathscr{M}} \sum_{i\in \mathscr{T}} \sum_{j \in \mathscr{C}} \left [ (y^t_i - y^c_j) a_{i,j} \right ]^2 - [ \sum_{i\in \mathscr{T}}\sum_{j \in \textcolor{black}{\mathscr{C}}} (y^t_i - y^c_j)a_{i,j}  ]^2 \  \text{S.t.: Constraints \eqref{z:angn1} to \eqref{z:bin}, and \eqref{min_case1_lin}}  \label{min_case1_coup2}
\end{align}

The above Quadratic Integer Program (QIP) \eqref{min_case1_coup2} can be solved with any commercial MIP solver. Using the solution of \eqref{min_case1_coup2} and a pre-specified $n$, we can calculate the optimal solution $\gamma^*$ by solving the quadratic equation \eqref{gamma} for $\gamma$. \textcolor{black}{As the condition for case 1 of the minimization problem is $\gamma \geq 0$, the $\gamma^*$ would be minimum non-negative solution of $\gamma$ from equation \eqref{gamma}.} Similarly, we can develop QIPs for case 1 and case 2 of the both minimization and maximization problems (see table \ref{tab:QIP}). The details on QIP formulation is provided in \ref{app:3}. Case 3 for both the problems results in a linear feasibility problem.  
\begin{align}
\label{gamma}
   ( \frac{n\gamma^2 }{n+\gamma^2} )\sum_{i\in \mathscr{T}} \sum_{j \in \mathscr{C}}  [ (y^t_i-y^c_j) a_{i,j}  ]^2 = [ \sum_{i\in \mathscr{T}}\sum_{j \in \mathscr{C}} (y^t_i-y^c_j)a_{i,j}  ]^2
\end{align}

\begin{table}[htbp]
  \begin{center}
  \caption{Quadratic integer programs (QIPs) for the robust Z-test cases. All the QIPs are subject to constraints \eqref{z:angn1} to \eqref{z:bin} and the case constraints.}
  {\small
  \adjustbox{max width=\textwidth}{
    \begin{tabular}{|c|c|c|}
    \hline
    \multicolumn{3}{|c|}{Minimization problems} \\
    \hline
    \multicolumn{1}{|l|}{Case} & \multicolumn{1}{l|}{Case Constraints} & \multicolumn{1}{l|}{QIP} \\
    \hline
       1   &    $\gamma \geq 0$, $\sum_{i\in \mathscr{T}}\sum_{j \in \mathscr{C}} (y^t_i - y^c_j)a_{i,j} \geq 0$   & $\max_{a_{i,j} \in \mathscr{M}} \sum_{i\in \mathscr{T}} \sum_{j \in \mathscr{C}} \left [ (y^t_i - y^c_j) a_{i,j} \right ]^2 - [ \sum_{i\in \mathscr{T}}\sum_{j \in \mathscr{C}} (y^t_i - y^c_j)a_{i,j}  ]^2 $ \\
    \hline
       2   &    $\gamma \leq 0$, $\sum_{i\in \mathscr{T}}\sum_{j \in \mathscr{C}} (y^t_i - y^c_j)a_{i,j} \leq 0$   &  $\max_{a_{i,j} \in \mathscr{M}} [ \sum_{i\in \mathscr{T}}\sum_{j \in \mathscr{C}} (y^t_i - y^c_j)a_{i,j}  ]^2 -  \sum_{i\in \mathscr{T}} \sum_{j \in \mathscr{C}} \left [ (y^t_i - y^c_j) a_{i,j} \right ]^2 $\\
    \hline
    \multicolumn{3}{|c|}{Maximization problems} \\
    \hline
        1  &   $\gamma \geq 0 $, $ \sum_{i\in \textcolor{black}{\mathscr{T}}} \sum_{j \in \textcolor{black}{\mathscr{C}}} [(\textcolor{black}{y^t_i} -\textcolor{black}{y^c_j}) a_{i,j}] \geq 0 $    &  $\max_{a_{i,j} \in \mathscr{M}} [ \sum_{i\in \mathscr{T}}\sum_{j \in \mathscr{C}} (y^t_i - y^c_j)a_{i,j}  ]^2 -  \sum_{i\in \mathscr{T}} \sum_{j \in \mathscr{C}} \left [ (y^t_i - y^c_j) a_{i,j} \right ]^2 $ \\
    \hline
        2  &   $\gamma \leq 0 $, $ \sum_{i\in \textcolor{black}{\mathscr{T}}} \sum_{j \in \textcolor{black}{\mathscr{C}}} [(\textcolor{black}{y^t_i} -\textcolor{black}{y^c_j}) a_{i,j}] \leq 0 $    & $\max_{a_{i,j} \in \mathscr{M}} \sum_{i\in \mathscr{T}} \sum_{j \in \mathscr{C}} \left [ (y^t_i - y^c_j) a_{i,j} \right ]^2 - [ \sum_{i\in \mathscr{T}}\sum_{j \in \mathscr{C}} (y^t_i - y^c_j)a_{i,j}  ]^2 $ \\
    \hline
    \end{tabular}}%
    }
  \label{tab:QIP}%
  \end{center}
\end{table}%

Even though we converted the general nonlinear integer optimization problems to QIPs, it is still computationally expensive to solve practical sized problems. Therefore, in the following, we develop greedy algorithms to solve the QIPs in table \ref{tab:QIP}. We propose three algorithms for solving both the minimization and maximization problem. First, in Algorithm \ref{alg:aa1}, we propose a greedy scheme for efficiently solving Case 1 and Case 2 of the minimization problem. Second, in Algorithm \ref{alg:max}, we extend a variant of this setup to maximization problem for Case 1 and Case 2. The Case 3 of both minimization and maximization problems are discussed in algorithm \ref{alg:aahung}. Finally, in Algorithm \ref{alg:master}, we provided a combined framework for solving the Z-test problem using Algorithms \ref{alg:aa1}-\ref{alg:aahung}. Before we delve into the details of the algorithms, first let us define the following constants for algorithm \ref{alg:aa1}:
\begin{align}
    \text{Case 1:} \ \ \alpha_1 = 1, \ \alpha_2 = 0, \quad \quad  \text{Case 2:} \ \ \alpha_1 = 0, \ \alpha_2 = -1.
\end{align}
Moreover, for all $i \in \mathscr{T}, \ j \in \mathscr{C} $ we set, $\Delta_{ij} = [y^t_i-y^c_j] \odot \mathbf{D}_{ij}$ and construct the list $\Upsilon$ as follows: 
\begin{align}
\label{sort}
   & \Upsilon[r, (i_r,j_r)] = (\Upsilon)_r = \Delta_{i_rj_r}; \quad  \Upsilon[r, (i_r,j_r)] \leq \Upsilon[r+1, (i_{r+1},j_{r+1})], \ \Delta_{i_rj_r} \neq 0 \  \ \forall r \in \|\mathbf{D}\|_0
\end{align}
This list will be used as input list $\Upsilon$ in Algorithms \ref{alg:aa1} and \ref{alg:max}.
\begin{algorithm}[ht!]
\caption{$\gamma = \mathcal{G}_1(\Upsilon, n, \alpha_1, \alpha_2)$}
\label{alg:aa1}
\begin{algorithmic}
\STATE{Initialize $  k \leftarrow 0, \ \epsilon_k \leftarrow \mathbf{0}$, $l \leftarrow  |\Upsilon|$; }
\WHILE{$k \leq \alpha_1\ceil[\big]{\frac{n}{2}} + |\alpha_2| n$ \textbf{or} $\epsilon_k > 0, \ $}
\STATE{\textbf{If} $ \alpha_2 \Upsilon[1, (i_1,j_1)] + \alpha_1 \Upsilon[l, (i_l,j_l)] < 0$}
\STATE{\quad \quad \textbf{Stop}; No feasible solution.}
\STATE{\textbf{Else} }
\STATE{\quad \quad \textbf{If} $|\Upsilon[1, (i_1,j_1)]| - \alpha_2 \Upsilon[1, (i_1,j_1)] -\alpha_1 \Upsilon[l, (i_l,j_l)] \geq 0$ \\
\quad \quad \quad \quad Assign $a_{i_l,j_l} = 1 $.}
\STATE{\quad \quad \textbf{Else} \\
\quad \quad \quad \quad Assign $a_{i_r,j_r} = 1 $ such that $|\Upsilon[1, (i_1,j_1)]+\Upsilon[r, (i_r,j_r)]|$ is minimized.}
\STATE{\quad \quad Denote the assigned entry as $a_{i_{\bar{r}}j_{\bar{r}}}$ and remove the entries of the list that contains indices $i_{\bar{r}}$ or $j_{\bar{r}}$.}
\STATE{ \quad \quad  $ \epsilon_k \leftarrow \epsilon_k+ \Upsilon[\bar{r}, (i_{\bar{r}},j_{\bar{r}})]$;}
\STATE{\textbf{If} $\alpha_1 = 1$ \\
\quad \quad Assign $a_{i_p, j_p} =1 $ such that $\Upsilon[p, (i_p,j_p)] + \Upsilon[\bar{r}, (i_{\bar{r}},j_{\bar{r}})] \geq 0$ and is minimized. Remove the entries that contains the indices $i_p$ or $j_p$ from the list.}
\STATE{\quad \quad  $ l \leftarrow l-2$;}
\STATE{\quad \quad  $k \leftarrow k+1$;}
\STATE{\textbf{Else}}
\STATE{\quad \quad $ l \leftarrow l-1$;}
\STATE{\quad \quad  $k \leftarrow k+1$;}
\ENDWHILE
\RETURN $\gamma = \alpha_2 \gamma_{\min} + \alpha_1 \gamma_{\max} $; \  $\gamma_{\min}$ and $\gamma_{\max}$ are respectively the minimum and maximum roots of equation \eqref{gamma}.
\end{algorithmic}
\end{algorithm}

In a similar fashion, the following constants will be used in Algorithm \ref{alg:max}:
\begin{align}
    \text{Case 1:} \ \ \beta_1 = -1, \ \beta_2 = 0, \quad \quad  \text{Case 2:} \ \ \beta_1 = 0, \ \beta_2 = 1.
\end{align}

\begin{algorithm}[ht!]
\caption{$\gamma = \mathcal{G}_2(\Upsilon, n, \beta_1, \beta_2)$}
\label{alg:max}
\begin{algorithmic}
\STATE{Initialize $  k \leftarrow 0, \ \epsilon_k \leftarrow \mathbf{0}$, $l \leftarrow  |\Upsilon|$; }
\WHILE{$k \leq \beta_2\ceil[\big]{\frac{n}{2}} + |\beta_1| n$ \textbf{or} $\epsilon_k > 0, \ $}
\STATE{\textbf{If} $ \beta_1 \Upsilon[1, (i_1,j_1)] + \beta_2 \Upsilon[l, (i_l,j_l)] < 0$}
\STATE{\quad \quad \textbf{Stop}; No feasible solution.}
\STATE{\textbf{Else} }
\STATE{\quad \quad \textbf{If} $(1+\beta_1)|\Upsilon[1, (i_1,j_1)]| -\beta_2 \Upsilon[l, (i_l,j_l)] \leq 0$ \\
\quad \quad \quad \quad Assign $a_{i_l,j_l} = 1 $.}
\STATE{\quad \quad \textbf{Else} \\
\quad \quad \quad \quad Assign $a_{i_r,j_r} = 1 $ such that $\Upsilon[l, (i_l,j_l)]-\Upsilon[r, (i_r,j_r)]$ is maximized.}
\STATE{\quad \quad Denote the assigned entry as $a_{i_{\bar{r}}j_{\bar{r}}}$ and remove the entries of the list that contains indices $i_{\bar{r}}$ or $j_{\bar{r}}$.}
\STATE{ \quad \quad  $ \epsilon_k \leftarrow \epsilon_k+ \Upsilon[\bar{r}, (i_{\bar{r}},j_{\bar{r}})]$;}
\STATE{\textbf{If} $\beta_2 = 1$ \\
\quad \quad Assign $a_{i_p, j_p} =1 $ such that $\Upsilon[p, (i_p,j_p)] - |\Upsilon[\bar{r}, (i_{\bar{r}},j_{\bar{r}})]| $ and is maximized. Remove the entries that contains the indices $i_p$ or $j_p$ from the list.}
\STATE{\quad \quad  $ l \leftarrow l-2$;}
\STATE{\quad \quad  $k \leftarrow k+1$;}

\STATE{\textbf{Else}}
\STATE{\quad \quad $ l \leftarrow l-1$;}
\STATE{\quad \quad  $k \leftarrow k+1$;}

\ENDWHILE
\RETURN $\gamma = \beta_2 \gamma_{\min} + \beta_1 \gamma_{\max} $; \  $\gamma_{\min}$ and $\gamma_{\max}$ are respectively the minimum and maximum roots of equation \eqref{gamma}.
\end{algorithmic}
\end{algorithm}

In Algorithm \ref{alg:aa1}, we consider Case 1 and Case 2. For Case 3, we can design an exact method by leveraging the assignment problem structure. Note that for this case, the nonlinear feasibility problem becomes a linear feasibility problem with constraint $\sum_{i\in \mathscr{T}} \sum_{j \in \mathscr{C}} [(y^t_i - y^c_j) a_{i,j}] \leq 0 $ along with necessary assignment constraints. Solving this linear feasibility problem is similar to solving an assignment problem where the assignment cost of assigning the $i^{th}$ treated sample to the $j^{th}$ control sample is $y^t_i - y^c_j$. Using this idea, we propose the \textcolor{black}{algorithm \ref{alg:aahung}} to solve the Case 3 of the minimization problem. 
\textcolor{blue}{
\begin{algorithm}[ht!]
\caption{$\gamma = \mathcal{G}_3(\Delta, n)$}
\label{alg:aahung}
\begin{algorithmic}
\STATE{\textbf{Step 1}: Put a sufficiently large number $\mathcal{M}$ where assignments are not possible. Convert the negative entries of $\Delta$ to non-negative by adding $\min \Delta_{i,j}$ to all the elements. Assume the new cost matrix is $\Delta^{'}$.} 
\STATE{\textbf{Step 2}: Solve the assignment problem for cost matrix $\Delta^{'}$ with Hungarian/Munkre's algorithm to minimize (for the minimization problem)/maximize (for the maximization problem).}
\STATE{\textbf{Step 3}: For a given $n$, take the first $n$ assignments achieved in Step 2 and calculate the total cost of the assignment using the original cost matrix $\Delta$. Set, $\epsilon^k =$ total assignment cost of $ n$ pairs. If $\epsilon^k \leq 0$ (maximization, $\epsilon^k \geq 0$) then, $\gamma^* = 0$, otherwise no feasible solution is possible.}
\end{algorithmic}
\end{algorithm}
}

For case 3 of the maximization problem, we follow the same scheme except in the maximization problem, instead of minimizing we maximize the assignment cost.

\begin{algorithm}[ht!]
\caption{$\gamma^* = \mathcal{A}(\Delta, n)$}
\label{alg:master}
\begin{algorithmic}
\STATE{Sort the list $\Delta$ following \eqref{sort}.}
\STATE{For the minimization problem:}
\STATE{Solve $ \gamma = \mathcal{G}_1(\Upsilon, n, 0, -1)$}
\STATE{\quad \textbf{If} No feasible solution.}
\STATE{\quad \quad Solve $\gamma = \mathcal{G}_3(\Delta, n)$}
\STATE{\quad \quad \quad \textbf{If} No feasible solution.}
\STATE{\quad \quad \quad \quad Solve $ \gamma = \mathcal{G}_1(\Upsilon, n, 1, 0)$}
\STATE{\quad \quad \quad \quad \quad \textbf{If} No feasible solution.}
\STATE{\quad \quad \quad \quad \quad \quad $n$ pairs are not possible.}
\STATE{\quad \quad \quad \textbf{Else} return the current $\gamma$ as optimal.}
\STATE{\quad \textbf{Else} return the current $\gamma$ as optimal.}
\STATE{For the maximization problem:}
\STATE{Solve $ \gamma^* = \mathcal{G}_2(\Upsilon, n, -1, 0)$}
\STATE{\quad \textbf{If} No feasible solution.}
\STATE{\quad \quad Solve $\gamma = \mathcal{G}_3(\Delta, n)$}
\STATE{\quad \quad \quad \textbf{If} No feasible solution.}
\STATE{\quad \quad \quad \quad Solve $ \gamma^* = \mathcal{G}_2(\Upsilon, n, 0, 1)$}
\STATE{\quad \quad \quad \quad \quad \textbf{If} No feasible solution.}
\STATE{\quad \quad \quad \quad \quad \quad $n$ pairs are not possible.}
\STATE{\quad \quad \quad \textbf{Else} return the current $\gamma$ as optimal.}
\STATE{\quad \textbf{Else} return the current $\gamma$ as optimal.}
\RETURN $\gamma^* $.
\end{algorithmic}
\end{algorithm}

\subsection{Properties of the greedy algorithms}
As we are making assignments in a greedy way, we may achieve a local optimal solution instead of a global optimal solution. In addition, we may have a fewer number of matched pairs than an exact method, as greedy approach ignores the combinatorial nature of the problem. In this section, we discuss properties of the algorithms \textcolor{black}{that} show that our algorithms can achieve a global optimal solution in some cases\textcolor{black}{,} and for specific matching restrictions, we will have the same number of pairs as an exact method. In addition, we discuss the time complexity of the proposed algorithms.


\begin{proposition}
\label{prop:1}
If the set of good matches $\mathscr{M}$ are identified through exact matching then, $\mathbf{\Delta}$ can be partitioned into a set of disjoint matrices $\{\Delta^1, \Delta^2, \cdots, \Delta^r, \cdots\}$ and each row in $\Delta^r$ is identical.
\end{proposition}

\begin{proof}
In exact matching, a treated sample $t_i$ can be matched to a control sample $c_{j_1}$ if covariate vector $\mathbf{X}_i^t = \mathbf{X}_{j_1}^c$. In the matrix $D$ (representing the set of good matches $\mathscr{M}$), the entry $d_{i,j_1} := 1$, if $\mathbf{X}_i^t = \mathbf{X}_{j_1}^c$, otherwise $d_{i,j_1} := 0$. Let's assume, treated sample $t_i$ can be matched the set of control sample $c^{'} = \{ c_{j_1}, c_{j_2},c_{j_3}, \cdots, c_{j_q}\}$. As the matching is done exactly, we can say that $\mathbf{X}_i^t = \mathbf{X}_{j_1}^c = \mathbf{X}_{j_2}^c = \mathbf{X}_{j_3}^c = \cdots = \mathbf{X}_{j_q}^c$. So, $d_{i,j_l} := 1$, $\forall l \in \{1,2,...,q\}$ and the rest of the values in $i^{th}$ row of the matrix $D$ are zero.  Now, assume treated unit $t_h$ is a good match to a control sample $c_k \in c^{'}$ then, by definition of exact match, $t_h$ is a good match for all the control samples in $c^{'}$. Hence, the vector $d_{h,j_l}$ will be identical to $d_{i,j_l}$. Now, by reorganizing the rows of $D$, we can partition it into disjoint matrices where rows within each of the matrices are identical. As $\mathbf{\Delta} = S \odot D$, it can be partitioned into disjoint matrices $\{\Delta^1, \Delta^2, \cdots, \Delta^r, \cdots\}$.
\end{proof}

From proposition \ref{prop:1}, we see that the $\mathbf{\Delta}$ can be partitioned into disjoint matrices and rows within a partitioned matrix are identical. So, by making greedy assignments, we will not lose other possible assignments at any point in the future.

\begin{proposition}
\label{prop:2}
If $(y_i^t - y_j^c)a_{i,j}$ is considered as the assignment cost of assigning treated unit $i$ to control unit $j$ where $a_{i,j} \in \{0,1\}$ is an assignment variable, then the total assignment cost $\sum_{i \in \mathscr{T}} \sum_{j \in \mathscr{C}} (y_i^t - y_j^c)a_{i,j}$ is independent of the order of assignment in $\Delta^r$.
\end{proposition}

\begin{proof}
From proposition 1, we know that in $\Delta^r$ each row has non-zero elements in same columns and respective partition of $D$ has identical rows. Now assume a possible assignment $\mathcal{A}_{1} = \{a_{1,1} = 1, a_{2,2} = 1, \cdots, a_{m,m} = 1, \cdots \}$. The total cost of assignment of $\mathcal{A}_1$ is  the following:
\begin{align}
   \sum_{i \in \mathscr{T}} \sum_{j \in \mathscr{C}} (y_i^t - y_j^c)a_{i,j} & = (y_1^t - y_1^c) + (y_2^t - y_2^c) + \cdots + (y_m^t - y_m^c) + \cdots \\
   & = (y_1^t + y_2^t + \cdots + y_m^t + \cdots) - (y_1^c + y_2^c + \cdots + y_m^c + \cdots) \label{indep}
\end{align}
Now, from the above expression \eqref{indep}, we see that the total cost is independent of the pair information $a_{i,j}$, i.e., which treated unit is paired with which control unit. Therefore, in each partition of $\mathbf{\Delta}$, the total cost of the assignment is independent of the order of pair assignment.
\end{proof}

\begin{proposition}
\label{prop:3}
The greedy scheme proposed in algorithm \ref{alg:aa1} provides an optimal solution of case 2 of the minimization problem when $n$ pairs are matched, there are at least $n$ possible pairs with $\Delta_{ij} \leq 0$, and $\mathscr{M}$ is identified with exact matching. 

\end{proposition}

\begin{proof}
In case 2 of the minimization problem, we are trying to maximize $[ \sum_{i\in \mathscr{T}}\sum_{j \in \mathscr{C}} (y_i^t - y_j^c)a_{i,j} ]^2  - \sum_{i\in \mathscr{T}} \sum_{j \in R} [ (y_i^t - y_j^c) a_{i,j}  ]^2$ with constraints \eqref{min_case2_lin}, and \eqref{z:angn1}-\eqref{z:bin}. This quantity is increasing in $|\sum_{i\in \mathscr{T}}\sum_{j \in \mathscr{C}} (y_i^t - y_j^c)a_{i,j}|$ when the difference between the outcomes $(y_i^t - y_j^c)a_{i,j}$ of all pairs are of the same sign. From proposition \ref{prop:2}, we can see that in each $\Delta^r \in \mathbf{\Delta}$ the cost $\sum_{i\in \mathscr{T}}\sum_{j \in \mathscr{C}} (y_i^t - y_j^c)a_{i,j}$ is independent of the order of assignment. Therefore, the greedy assignment will produce an optimal solution for each $\Delta^r \in \mathbf{\Delta}$. On the other hand, from proposition \ref{prop:1}, we see that each $\Delta^r \in \mathbf{\Delta}$ is disjoint when a good set of match $\mathscr{M}$ is identified with exact matching. Hence, the optimal solution of each disjoint set of possible assignments will produce an optimal solution of the complete assignment problem. 
\end{proof}

Case 1 of the maximization problem follows the same scheme as the case 2 of the minimization problem. Hence, the result in proposition \ref{prop:3} is also valid for case 1 of the maximization problem.

\begin{proposition}
\label{prop:heu}
Denote, $d = \max\left\{\ln{\|D\|_0},n\right\} $. Then the proposed greedy algorithms have running time complexity of $\mathcal{O} \left( d \|D\|_0  \right)$ and have storage cost of $\mathcal{O}(\|D\|_0 )$.
\end{proposition}

\begin{proof}
Note that if we run sorting algorithms such as \textit{HeapSort}, \textit{MergeSort} on a list of $n$ entries, the best worst-case running time complexity that can be achieved is $\mathcal{O}(n \ln{n})$. In our proposed schemes, we run a sorting algorithm on the respective $\mathbf{\Delta}$ matrix and then we input the sorted data on the main algorithms (Algorithm \ref{alg:aa1} and \ref{alg:max}). As both algorithms run in at most $n$ loops (i.e., $\alpha_1\ceil[\big]{\frac{n}{2}} + |\alpha_2| n, \beta_2\ceil[\big]{\frac{n}{2}} + |\beta_1| n \leq n$), we can calculate the running time complexity as follows:
\begin{align*}
    T_1(n) = n \left[1+1+\|D\|_0+ 2+ \|D\|_0\right] = \mathcal{O} \left(n \|D\|_0\right)
\end{align*}
Then, considering the time complexity of the sorting scheme, we can find the total time complexity of the proposed scheme as follows:
\begin{align*}
  T(n) & :=  \text{Sorting run time} + \text{Running time of Algorithm \ref{alg:aa1}}\text{ (or Algorithm }\ref{alg:max})\\
  & = \mathcal{O} \left( \|D\|_0 \ln{\|D\|_0}\right) + \mathcal{O} \left(n \|D\|_0 \right) = \mathcal{O} \left( \max\{n,  \ln{\|D\|_0}\} \|D\|_0  \right) = \mathcal{O} \left( d \|D\|_0  \right) 
\end{align*}
Here, we used the fact that we run the sorting scheme on the nonzero elements of the matrix $\mathbf{\Delta}$ which has a total of $\|D\|_0$ entries. It's easy to check that both schemes have a total storage cost of $\mathcal{O}(\|D\|_0 )$ as throughout the scheme, we only need to store at-most $\|D\|_0$ entries.
\end{proof}

\begin{proposition}
\label{prop:heu2}
The proposed heuristic for Case 3 runs in strongly polynomial time. Moreover, it will provide us an exact solution.
\end{proposition}

\begin{proof}
Since we are using the \textit{Hungarian-Munkres} algorithm for solving the main problem, we can calculate the time complexity of the proposed heuristic as follows:
\begin{align*}
    T(n) = \mathcal{O}(n) + \mathcal{O}(n^3) = \mathcal{O}(n^3)
\end{align*}
Here, we used the complexity result of \textit{Hungarian-Munkres} algorithm \citep{hungarian} and the initial time complexity of $\mathcal{O}(n)$. Furthermore, as the \textit{Hungarian-Munkres} scheme provides an exact solution, the proposed heuristic will also provide an exact solution.
\end{proof}

\section{Numerical Experiments}
In this section, we apply the proposed greedy algorithms to test causal hypotheses from three real-world datasets of varying size. We also compare the solution of the proposed algorithms to the Quadratic Integer Programming (QIP) models in table \ref{tab:QIP} solved with Gurobi 9.0.2 \citep{gurobi2020}  \textcolor{black}{and the ILP-based heuristic proposed by \citet{morucci2018hypothesis}.} It is important to note that we implemented a simpler but computationally less efficient version of the proposed algorithm as we use unsorted $\mathbf{\Delta}$ and find the maximum or minimum $\Delta_{i,j}$ at each iteration. \textcolor{black}{After calculating the maximum and minimum $Z(\mathbf{a})$ using the proposed greedy algorithms and alternative methods, we convert them into P-values using the relations in equation \eqref{p_max} and \eqref{p_min} as P-values are commonly used to make inference from the statistical test. }
\textcolor{black}{
\begin{align}
    \textrm{P-value}_{max} = \argmax_{\mathbf{a}\in \mathscr{M}} [1- \phi(Z(\mathbf{a}))] = 1 - \phi(\argmin_{\mathbf{a}\in \mathscr{M}}  Z(\mathbf{a})) \label{p_max}\\
    \textrm{P-value}_{min} = \argmin_{\mathbf{a}\in \mathscr{M}} [1- \phi(Z(\mathbf{a}))] = 1 - \phi(\argmax_{\mathbf{a}\in \mathscr{M}}  Z(\mathbf{a})) \label{p_min}
\end{align}
Here, $\phi(Z(\mathbf{a}))$ represents the area under the standard normal distribution curve to the right of $Z(\mathbf{a})$. For all the experiments, we consider the level of significance $\alpha = 0.05$ as a rule of thumb to make robust inference.} Even so, our experiments show that the proposed algorithm is scalable and can handle very large-scale problems while state-of-the-art commercial solvers either provide a worse solution or cannot solve problems of moderate to large sizes. While the experiments show interesting causal insights, we use them to demonstrate the effectiveness of our algorithms. All the experiments are performed in a Dell Precision 7510 workstation with Intel Core i7-6820HQ CPU running at 2.70 GHz, and 32GB RAM.

\subsection{Effect of fly ash on strength of concrete}
Fly ash, a by-product of thermal power plants, is a common element in producing concrete \citep{flyash}. In this experiment, we hypothesize that fly ash has zero effect on concrete's compressive strength. We use the Concrete Compressive Strength Dataset \citep{concrete:data} to test the causal hypothesis using the robust Z-test. The dataset has 1030 instances and 9 attributes. The control group includes 529 samples with no fly ash component wherein the treatment group has 501 samples with at least 24.5 $kg/m^3$ fly ash. We perform the matching operation on seven pre-treatment covariates. A treated unit $i$ is a good match with control unit $j$ (i.e., $d_{i,j} = 1)$ if their differences in Cement, Blast Furnace Slag, and Water are less or equal to 30, the difference in Superplasticizer is less or equal to 20, fine and coarse aggregate is less or equal to 50, and age is less or equal to 5. The outcome is concrete's compressive strength. The matching process resulted in a group of 68 treated samples that can be matched with 60 control samples where many treated samples have multiple pair assignment options. The $D$ matrix has 146 non-zero entries which indicate that we need to solve a nonlinear optimization problem with 146 binary variables.
 
Table \ref{tab:concrete} shows the maximum and minimum Z-values achieved for different numbers of pairs ($n$) with the greedy algorithms (GA) and solving the QIPs in table \ref{tab:QIP} with Gurobi 9.0.2. For QIP, we used the stopping criteria as 500 seconds time limit or a 2\% optimality gap. As we can see that the QIP provides better solutions compared to the proposed greedy algorithm. However, the greedy approach takes a fraction of a second while QIP takes more than 500 seconds. For the minimization problem with the number of pairs higher than 24, the QIP achieves a 2\% optimality gap in a fairly short duration but, after that, it takes days to reach optimality. The resulting P-values are presented in Figure \ref{fig:concrete}. \textcolor{black}{In Figure \ref{fig:concrete}, we also include the result using the ILP-based heuristic proposed by \citet{morucci2018hypothesis}. For the number of pairs greater than 26, all three algorithms achieve P-values less than 0.05 and after 38 pairs, both the greedy algorithm and ILP-based heuristic did not find any pairs. As both the maximum and minimum P-values coincide in Figure \ref{fig:concrete}, we achieve an absolute robust test according to the definition \ref{defrobust_test}. Even though the ILP-based heuristic and QIP perform better in terms of solution quality, for a larger number of pairs, the greedy algorithm achieves similar quality solutions. All three algorithms reject the hypothesis of the zero treatment effect which supports the traditional knowledge of fly ash's positive effect on concrete strength \citep{concrete:data}. This result shows that the proposed greedy approach (GA) can achieve the same conclusion as the ILP-based heuristics and QIP, but in a significantly smaller amount of time.} 
\begin{table}[htbp]
  \centering
  \caption{Comparison of Greedy Algorithm (GA) and QIP solved with Gurobi with Concrete Compressive Strength dataset. QIP is solved with stopping criteria as 500 seconds time limit or 2\% optimality gap.}
 \adjustbox{max width=\textwidth}{

    \begin{tabular}{|c|c|c|c|c|c|c|c|c|}
    \hline
    \multicolumn{1}{|c|}{\multirow{2}[4]{*}{$n$}} & \multicolumn{4}{c|}{Maximum Z} & \multicolumn{4}{c|}{Minimum Z} \\
\cline{2-9}          & \multicolumn{1}{l|}{Z with GA} & \multicolumn{1}{l|}{Z with QIP} & \multicolumn{1}{l|}{Time GA} & \multicolumn{1}{l|}{Time QIP} & \multicolumn{1}{l|}{Z with GA} & \multicolumn{1}{l|}{Z with QIP} & \multicolumn{1}{l|}{Time GA} & \multicolumn{1}{l|}{Time QIP} \\
    \hline
    20    & 12.257 & 13.343 & 0.019 & 500  & 0.927 & 0.4346 & 0.028 & 500 \\
    \hline
    22    & 12.446 & 13.555 & 0.016 & 500  & 1.2246 & 0.626 & 0.019 & 500 \\
    \hline
    24    & 12.584 & 13.78 & 0.019 & 500  & 1.5594 & 1.1292 & 0.022 & 500 \\
    \hline
    26    & 12.645 & 15.272 & 0.021 & 500  & 1.993 & 1.6151 & 0.023 & 471 \\
    \hline
    28    & 12.564 & 15.195 & 0.021 & 500  & 2.424 & 2.0858 & 0.021 & 230 \\
    \hline
    30    & 12.129 & 15.232 & 0.029 & 500  & 2.8573 & 2.5886 & 0.030 & 130 \\
    \hline
    32    & 11.624 & 14.79 & 0.019 & 500  & 3.3077 & 3.0315 & 0.023 & 50 \\
    \hline
    34    & 11.171 & 14.187 & 0.016 & 500  & 3.7772 & 3.525 & 0.029 & 51 \\
    \hline
    36    & 9.5952 & 13.567 & 0.019 & 500  & 4.2498 & 4.0181 & 0.037 & 61 \\
    \hline
    38    & 7.8861 & 12.948 & 0.024 & 500  & 4.7101 & 4.4831 & 0.027 & 45 \\
    \hline
    \end{tabular}}%
  \label{tab:concrete}%
\end{table}
\begin{figure}[h!]
    \begin{center}
    \includegraphics[scale = 0.6]{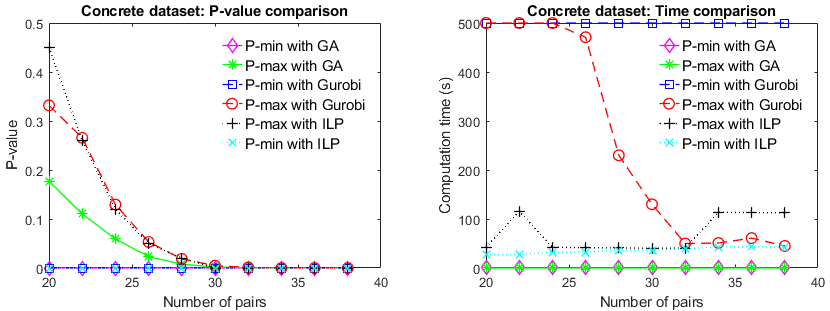}
    \caption{Comparison between Greedy Algorithm (GA), \textcolor{black}{ILP-based heuristic (ILP) proposed by \citet{morucci2018hypothesis}, and Gurobi solving the} QIP on P-values and computation time for the concrete dataset.}
    \label{fig:concrete}
   \end{center} 
\end{figure}

\subsection{Effect of misty weather on number of bike rentals}
In this experiment, we consider a slightly larger sized problem compared to the Concrete Compressive Strength dataset. We evaluate the effect of misty weather on the number of bike rentals. Our naive hypothesis is that there is no effect of mist on the rental bike usage. To test this hypothesis, we use the Bike-sharing dataset \citep{fanaee2014event} available in the UCI Machine learning repository. The dataset contains the daily count of bikes rented for 731 days, weather, and seasonal information of corresponding days between 2011 and 2012 in the capital bike-share system at Washington DC, USA. We consider 247 days as the treatment group\textcolor{black}{,} which had mist with a different combination of clouds. 463 days are considered as control\textcolor{black}{,} which consist\textcolor{black}{ed} of clear sky, few clouds, or partly cloudy days (without any mists). The seasonal information such as season, year, and workday were matched exactly. The weather variables such as temperature, wind speed, and humidity were matched if the differences were less or equal to 2, 6, and 6, respectively. If treated sample $i$ and control sample $j$ follows the above criteria, then $d_{i,j} = 1$, otherwise 0. This matching process produced a nonlinear optimization problem with 326 binary variables, more than double in size compared to the previous experiment. The robust test statistics achieved by the proposed algorithms and solving the QIPs of table \ref{tab:QIP} with Gurobi along with the computation time are provided in table \ref{tab:bike} and figure \ref{fig:bike}. We followed the QIP stopping criteria same as to the previous experiment.
\begin{table}[htbp]
  \centering
  \caption{Comparison of Greedy Algorithm (GA) and QIP solved with Gurobi with the bike-sharing dataset. QIP is solved with stopping criteria as 500 seconds time limit or a 2\% optimality gap. ``-'' implies no integer solution found within the time limit.}
  \adjustbox{max width=\textwidth}{
    \begin{tabular}{|c|c|c|c|c|c|c|c|c|}
    \hline
    \multicolumn{1}{|c|}{\multirow{2}[4]{*}{$n$}} & \multicolumn{4}{c|}{Maximum Z} & \multicolumn{4}{c|}{Minimum Z} \\
\cline{2-9}          & \multicolumn{1}{l|}{Z with GA} & \multicolumn{1}{l|}{Z with QIP} & \multicolumn{1}{l|}{Time GA} & \multicolumn{1}{l|}{Time QIP} & \multicolumn{1}{l|}{Z with GA} & \multicolumn{1}{l|}{Z with QIP} & \multicolumn{1}{l|}{Time GA} & \multicolumn{1}{l|}{Time QIP} \\
    \hline
    50    & 6.9736 & 6.7914 & 0.034 & 500   & -11.5334 & -10.1657 & 0.035 & 500 \\
    \hline
    55    & 6.0925 & 5.8509 & 0.032 & 500   & -11.2922 & -9.6567 & 0.034 & 500 \\
    \hline
    60    & 5.1916 & 4.8345 & 0.034 & 500   & -10.9022 & -9.1222 & 0.028 & 500 \\
    \hline
    65    & 4.2864 & 3.7268 & 0.036 & 500   & -10.5792 & -8.4047 & 0.036 & 500 \\
    \hline
    70    & 3.4461 & 2.6946 & 0.039 & 500   & -10.181 & -7.5596 & 0.036 & 500 \\
    \hline
    75    & 2.5065 & 1.0067 & 0.034 & 500   & -9.6883 & -6.2907 & 0.041 & 500 \\
    \hline
    80    & 1.4238 & 0.2036 & 0.043 & 500   & -9.1318 & -5.0902 & 0.036 & 500 \\
    \hline
    81    & 1.216 & -0.2893 & 0.039 & 500   & -9.0138 & -4.7659 & 0.036 & 500 \\
    \hline
    82    & 0.9255 & -0.5338 & 0.036 & 500   & -8.9001 & -4.4127 & 0.039 & 500 \\
    \hline
    83    & 0.6468 &    -   & 0.049 & 500   & -8.7803 &   -    & 0.050  & 500 \\
    \hline
    84    & 0.3713 &   -    & 0.039 & 500   & -8.6469 &  -     & 0.047 & 500 \\
    \hline
    85    & -1.2046 &   -    & 0.045 & 500   & -8.4748 &   -    & 0.036 & 500 \\
    \hline
    86    & -1.3929 &    -   & 0.037 & 500   & -8.2821 &  -     & 0.036 & 500 \\
    \hline
    87    & -1.6767 &   -    & 0.053 & 500   & -8.0448 &   -    & 0.046 & 500 \\
    \hline
    88    & -1.9539 &   -    & 0.043 & 500   & -7.7887 &   -    & 0.037 & 500 \\
    \hline
    \end{tabular}}%
  \label{tab:bike}%
\end{table}

From table \ref{tab:bike}, we can see that the greedy approach outperforms the QIP for all number of pairs. In fact, the QIP could not find an initial integer solution within the time limit for $n$ greater than 82. For $n$ between 50 to 82, the QIP solution improves marginally after 500 seconds. On the other hand, the greedy algorithm finds better solutions with a significantly lower amount of time for a number of matched pairs up to 88. 

\textcolor{black}{We also compare the results from the greedy algorithm with the ILP-based heuristic of \citet{morucci2018hypothesis} in Figure \ref{fig:bike}. For a fair comparison, we use the same matching algorithm and recommended heuristic settings as described in \citet{morucci2018hypothesis}. Figure \ref{fig:bike} shows that the greedy algorithm finds 88 pairs in the good set of matches and after 87 matched pairs, both maximum and minimum P-values are less than $\alpha = 0.05$. Therefore, we have an $\alpha-$robust test that fails to reject our hypothesis on the effect of mist on rental bike usage. ILP-based heuristic also achieves the $\alpha-$robust test; however, it finds 94 matched pairs in the dataset. As the ILP-based heuristic solves an integer linear program iteratively by bounding the sample standard deviation, it finds the maximum number of pairs. In contrast, the proposed algorithm assigns pair greedily so, it may select lesser number of pairs. Nonetheless, as we show in the proposition \ref{prop:3}, the greedy algorithm can achieve optimal solution and the maximum number of pairs if the matching is performed using exact matching methods. In terms of computation time, the greedy algorithm takes a fraction of a second wherein the ILP-based heuristic takes more than 9 seconds to solve an instance of the problem.     
}
\begin{figure}[h!]
    \begin{center}
    \includegraphics[scale = 0.6]{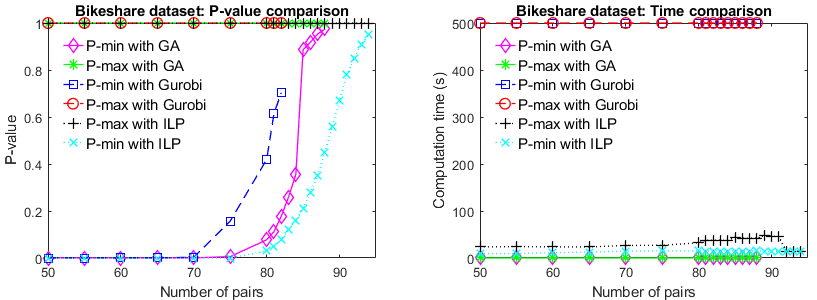}
    \caption{Comparison between Greedy Algorithm (GA), \textcolor{black}{ILP-based heuristic (ILP) proposed by \citet{morucci2018hypothesis}, and Gurobi solving the} QIP on P-values and computation time for the bike-sharing dataset.}
    \label{fig:bike}
   \end{center} 
\end{figure}

\subsection{Effect of product's location on price}
To show the scalability of the proposed greedy algorithm, we consider a very large-scale clickstream data of online shopping. We used the clickstream data from \citet{eshop} which contains information on clicks from online stores that are selling clothes for pregnant women. In this paper, we refer to the dataset as e-shop dataset. We hypothesize that if a store web page is split horizontally into two panels, top and bottom, the top part is as valuable as the bottom part and high price products are not always placed on the top panels. The dataset consists of 165,474 instances of clicks and each instance contains information on the product clicked on, time of the click, and origin of the IP address. The treatment group includes more than 35,000 click instances on the top-left panel products and the control group includes more than 27,000 clicks on products displayed on the bottom-left panel. We implemented exact matching on the pre-treatment covariates: month, order, country of IP address, product category, color, model photography, and page number within the website. For treated and control samples $i$ and $j$, respectively, $d_{i,j} = 1$ if two samples have exactly the same value on all the covariates. Product price is the outcome. Compared to the concrete strength and bike-sharing examples, e-shop data has a significantly large number of samples and produces a nonlinear optimization problem with more than 350,000 binary variables. The decision problem of this scale is almost impossible to solve with state-of-the-art commercial solvers. However, the proposed greedy algorithm can solve the problem in a reasonable amount of time considering the problem size.

Figure \ref{fig:e_shop} presents the maximum and minimum P-values achieved with the greedy algorithms for the robust Z-test and the time required for a different number of pairs. After 3,800 pairs, both the maximum and minimum P-values coincide at zero. So, we reject the hypothesis and conclude that the product location and price have a causal relation: high price products are placed on the top panel. From CPU time consumption in figure \ref{fig:e_shop}, we can see that the greedy algorithm takes significantly higher time at around 1300 seconds to solve the Z-test problems for e-shop data compared to the concrete strength and bike-sharing datasets. However, considering a nonlinear optimization problem with over 350,000 binary variables, the time consumption can be considered reasonable. \textcolor{black}{On the other hand, after running several hours, the ILP-based heuristic and QIP (solved with Gurobi) ran out of memory and could not find a solution. }
\begin{figure}[h!]
    \begin{center}
    \includegraphics[scale = 0.65]{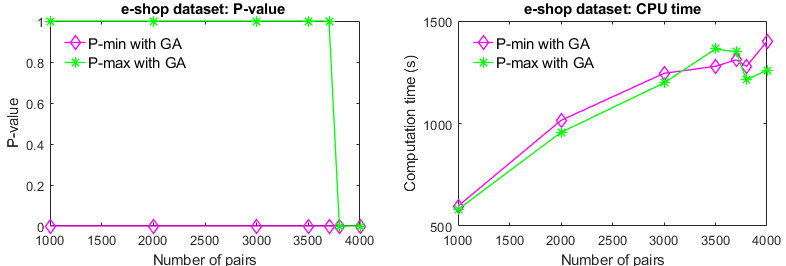}
    \caption{Comparison between Greedy Algorithm (GA) and \textcolor{black}{Gurobi solving the} QIP on P-values and computation time for the e-shop dataset.}
    \label{fig:e_shop}
   \end{center} 
\end{figure}

\textcolor{black}{Considering the results from three datasets, it is evident that both the ILP-based heuristic and solving the QIP with a commercial solver provide marginally better quality solutions, nonetheless, finds same inferential conclusion. 
For very small problems (i.e., problems similar to the Concrete Compresive Strength datasets), ILP-based heuristic can be a reasonable choice; however, in today's big data world, such small problems are highly unlikely in practice. In contrast, the greedy algorithm proposed in this paper provides same conclusion on the robust inference in a significantly smaller amount of time and are highly scalable for practical sized problems. 
}

\subsection{Inferential Implication}

\textcolor{black}{In general, the greedy algorithms proposed in this paper make local optimal choices. In this section, we investigate the inferential implication of such approximate solution choices by identifying how much deviation from the global optimal solution is allowable to preserve the actual inference of the robust hypothesis test. }

\textcolor{black}{
In Figure \ref{fig:implication}, we present three possible scenarios of P-value for the robust Z-test. P-value can be close to 0 or 1 (Figure \ref{fig:implication} (a) and (c), respectively) or somewhere in the middle (Figure \ref{fig:implication} (b)). Let's assume $Z_{max}^{greedy}$ is the solution of the greedy algorithm and $Z_{max}^{opt}$ is the global optimal solution of the maximization problem in any scenario. Now as the greedy algorithm will produce a sub-optimal solution then we can say $Z_{max}^{greedy} = Z_{max}^{opt} - gap$ where, $gap$ represents the optimality gap between two methods. Therefore, based on equations \eqref{p_max} and \eqref{p_min}, P-values with greedy algorithm and optimal solution will have the following relation: $P_{min}^{greedy} \geq P_{min}^{opt} $. Now, for a level of significance $\alpha$, the P-value with greedy algorithm ($P_{min}^{greedy}$) can deviate by the amount $\alpha$ without changing the conclusion on robustness. As the P-value is calculated from the standard normal distribution curve, we can establish a connection between $\alpha$ and $gap = Z_{max}^{opt} - Z_{max}^{greedy} $ using equation \eqref{p_max} and \eqref{p_min}: how much deviation on the $Z(\mathbf{a})$ we can allow without altering the inference on robustness. We can derive a similar relation for $P_{max}$ from the minimization problem following the same argument.
}
\begin{figure}[h!]
    \begin{center}
    \includegraphics[scale = 0.5]{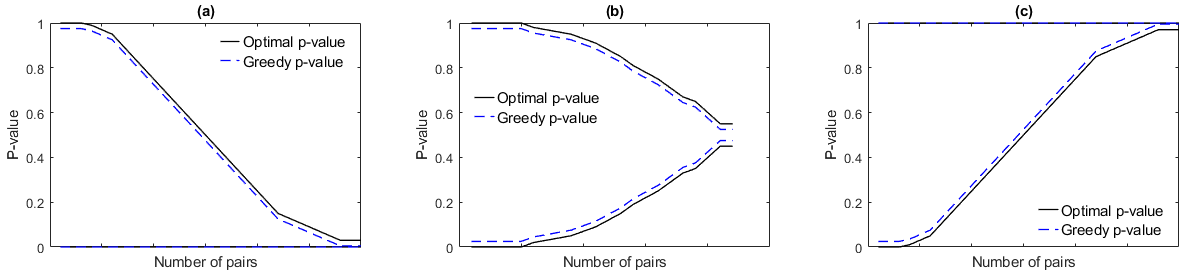}
    \caption{Different scenarios showing the gap between P-values obtained by an exact method and the greedy algorithm.}
    \label{fig:implication}
   \end{center} 
\end{figure}

\textcolor{black}{
Using $\alpha = 0.05$ as the rule of thumb and the probabilities from the standard normal distribution table, we can calculate the maximum allowable gap. For all the scenarios, if the optimal solutions of both maximization and minimization problems are very large (i.e., $Z \geq 4$ in Figure \ref{fig:implication} (a)) for maximum number of pairs ($n$) like the solutions of e-shop and Concrete Compressive Strength datasets, it would require more than 50\% optimality gap between the greedy and global optimal solutions to reverse the inference on robustness. This also applies to the situation when both problems have very small Z values (i.e., $Z \leq -4$ in Figure \ref{fig:implication} (c)) for maximum value of $n$ in the dataset as the area under the normal distribution curve is very small for $Z \geq |4|$. When $Z \leq |4|$, an approximate allowable $gap$ is $18\%$. In the case Concrete Compressive Strength study, we know the optimal solution for seven instances. For those seven instances, the average gap between the greedy and global optimal solution is around $11\%$ which is lower than the allowable gap. For the cases in Figure \ref{fig:implication} (a) and \ref{fig:implication} (c), we will have the same allowable gap due to the symmetry of the standard normal distribution curve. When both the maximum and minimum Z values deviate from optimality as shown in Figure \ref{fig:implication} (b), we can split the allowable gap into half for each of the problem. Despite the fact that the observed optimally gaps of the proposed greedy algorithms are well below the maximum allowable gap, our analysis shows that the proposed greedy algorithms will find an optimal solution (with zero gap) for restricted cases (i.e., when match pairs are constructed with exact matching). If all the variables in the dataset are categorical, we can apply an exact matching algorithm to find the good set of matches. In addition, we can take the same advantage with continuous data by converting them into categories by applying domain knowledge or any appropriate discretization algorithm.
}

\subsection{Practical Guideline}
\textcolor{black}{
In the numerical experiments, we show the test-statistics for a wide variety of sample sizes ($n$). Our purpose here is to show the level of uncertainty in the inference; however, in practice the experimenter’s goal is to find the robust inference. Hence, we do not need to consider such wide range of $n$. Instead, we can pick a suitable $n$ and increase (decrease) it until the problem become infeasible (or feasible) due to the constraint \eqref{feas:fixed_pair}: $\sum_{i \in \mathscr{T}} \sum_{j \in \mathscr{C}} a_{ij} = n$. In the good set of matches, multiple treated units can be a good match to a single control unit (or the other way around). As we are considering one-to-one pair assignment, we may not be able to assign pairs for all the treated or control samples. Therefore, if we keep increasing $n$, after certain number of pairs, the optimization problem will be infeasible. Conversely, if we select a large $n$ and decrease it, at certain number of pairs, the problem will be feasible. Unfortunately, identifying that specific value of $n$ is not possible without solving a combinatorial optimization problem.
In this case, an efficient approach can be starting with highest and lowest values of $n$ possible $(n_{max}, n_{min})$ and apply a binary search strategy until the problem becomes infeasible.
}

\textcolor{black}{We recommend starting the experiment with $n_{max}$ as the smallest number of treated or control samples in the good set of matches $D^{|\mathscr{T}| \times |\mathscr{C}|}$: $n_{max} = min(|\mathscr{T}|, |\mathscr{C}|)$ where $|\mathscr{T}|$ number of treated samples that are good matches for $|\mathscr{C}|$ number of control samples. The experimenter can set $n_{min}$ based on the number possible good matches in the dataset. If there are many non-zero entries available in $D^{|\mathscr{T}| \times |\mathscr{C}|}$, an experimenter can choose larger value of $n_{min}$ that can save significant computation time. 
For instance, in the experiment with Concrete Compressive Strength Dataset, we can start with $n_{max} = 60$ as the matching algorithms provided a group of 68 treated samples that can be matched with 60 control samples. If we apply a binary search strategy considering $n_{min} = 20\%$ of $n_{max}$, we can find a robust test with a single run of the proposed algorithm. Using the exact strategy for the Bikeshare dataset, we can find a robust solution by checking for only two different values of $n$. }

\section{Conclusion}
In this study, we investigate the robust causal hypothesis test, robust Z-test, from observational data with continuous outcomes and develop a unique computational framework which includes \textcolor{black}{a} novel reformulation technique and efficient algorithms. The robust Z-test produces nonlinear integer optimization problems that are difficult to solve for very small datasets where in today's big data world, the causal hypothesis test problems are becoming larger and larger. We reformulated the nonlinear optimization models of robust Z-test into feasibility problems. By leveraging the structure of the reformulation, we developed greedy algorithmic schemes that are very efficient and scalable. \textcolor{black}{The feasibility reformulation also allows us to pose the robust test problems as quadratic integer programming problems and for smaller datasets, we can use any commercial MIP solvers to achieve exact solution.}
Moreover, the proposed unique reformulation scheme can be used to model general nonlinear and quadratic optimization problems (i.e., Quadratic Assignment Problem) as feasibility problems. Apart from the scalability, we show that the greedy approaches achieve the global optimal solution in many cases. The effectiveness of the proposed algorithms is demonstrated with three real-world case studies of varying sizes and comparing the result with equivalent QIPs solved with exact method \textcolor{black}{and the ILP-based heuristic proposed by \citet{morucci2018hypothesis}}. Numerical experiments on the case studies reveal that the proposed algorithms achieve the same inference as the exact method for the small test case\textcolor{black}{;} however, \textcolor{black}{it} takes significantly less computational time. On the other hand, for moderate to large instances, the proposed algorithms significantly outperform \textcolor{black}{both methods}. For moderate\textcolor{black}{ly} size problems, our algorithm produces a better solution in a fraction of a second while the exact method struggles to find any integer solution in hundreds of seconds. With the availability of observational data and increasing use of causal inference in decision-making, our algorithms can be very effective in harnessing the power of big data in the decision-making process. \textcolor{black}{A major limitation of the greedy algorithms is that they provide local optimal solutions; hence, there is a chance of altering the robustness inference. Nonetheless, it would take significantly large optimally gap to change the conclusion on robustness and alternative methods cannot solve large-scale problems.}

As a future extension, we plan to use the feasibility formulation of the Z-test to develop several potential algorithms. First, we plan to develop an iterative projection-based algorithm to solve the feasibility problems with quadratic constraints and binary variables. Then, we intend to incorporate algorithmic acceleration schemes to further improve the efficiency of the iterative algorithm and ensure scalability.



\bibliography{mybibfile}

\newpage
\appendix
\setcounter{page}{1}


\section{Assumptions in Causal Inference}{\label{app:1}}
In this paper, we make the following assumptions which are commonly used in causal inference literature and are essential in estimating causal quantity from observational data \citep{rosenbaum1983central,stuart2010matching}.

\begin{assumption}
\label{assump:ignore}
\textup{(Conditional Independence)} $\exists$ a measured set of covariates $\mathbf{X} \in \mathcal{X}$ such that conditioning on $\mathbf{X}$, the potential outcomes and treatment status of samples are independent: $(Y_i^1, Y_i^0) \indep T_i | \mathbf{X}_i $, $\forall i \in \mathscr{S}$. 
\end{assumption}

\begin{assumption}
\label{assump:positivity}
\textup{(Positivity)} Given a set of measured covariate $\mathbf{X}_i$, each sample $i \in \mathscr{S}$ has a strictly positive probability of receiving treatment: $0< Pr(T_i = 1| \mathbf{X}) < 1,$ $\forall i \in \mathscr{S}$.
\end{assumption}

\begin{assumption}
\label{assump:sutva}
\textup{(SUTVA)} Treatment applied to sample $i$ does not effect the outcome of sample $j$: $Y_j^1 \indep T_i$ when $i \neq j$ and only one version of treatment exists: $Y_i^{(T_1,T_2,\cdots, T_N)} = Y_i^{T_i}$.
\end{assumption}

Assumption \ref{assump:ignore}, also known as \textit{strong ignorability} and \textit{selection on observables}, infers that the treatment assignment process on samples retained after controlling for $\mathbf{X}$ is `as good as random'. Hence, we can use the untreated samples as the counterfactual. Assumption \ref{assump:positivity} ensures that the treated and control group samples have overlap in their covariate distributions (i.e., common support) so that we are not extrapolating. Finally, assumption \ref{assump:sutva} implies that the samples do not interfere with each other's outcomes and every sample in the study receives treatment in the same way so that the way of receiving treatment does not confound its effect on the outcome. These assumptions ensure causal inference made from observational data with the matching method is valid and as unbiased as possible. 

\newpage

\section{Feasibility Formulation of the Maximization Problem} {\label{app:2}}
In this section, we discuss the feasibility reformulation for the maximization problem and the possible cases.

\subsection*{Cases for Maximization problem}

The quadratic constraint in the maximization problem is setup in the following way:
\begin{align}
    & \frac{1}{\sqrt{n}} \sum_{i\in \mathscr{T}}\sum_{j \in \mathscr{C}} (y^t_i -y^c_j)a_{i,j} \geq \gamma \times \sqrt{\frac{1}{n} \sum_{i\in \mathscr{T}} \sum_{j \in \mathscr{C}} [(y^t_i - y^c_j) a_{i,j}]^2 - (\frac{1}{n} \sum_{i\in \mathscr{T}}\sum_{j \in \mathscr{C}} (y^t_i - y^c_j)a_{i,j})^2 } \label{feas:max_root}\\
    & \frac{1}{n} (\sum_{i\in \mathscr{T}}\sum_{j \in \mathscr{C}} (y^t_i - y^c_j)a_{i,j})^2 \geq \gamma^2 \frac{1}{n} \sum_{i\in \mathscr{T}} \sum_{j \in \mathscr{C}} [(y^t_i - y^c_j) a_{i,j}]^2 - \gamma^2 (\frac{1}{n} \sum_{i\in \mathscr{T}}\sum_{j \in \mathscr{C}} (y^t_i - y^c_j)a_{i,j})^2 \label{feas:max}
\end{align}
We take the square of this inequality and consider additional constraints to validate the action. The resulting cases along with the constraints are presented below. The equivalent feasibility formulation of maximizing the Z-test model can be framed as the following binary-feasibility problem: $\exists \textrm{ a set of assignment } a_{i,j} \textrm{ so that } Z(\mathbf{a}) \geq \gamma$ with constraints \eqref{feas:quad_max}-\eqref{feas:bin_appen}.
\begin{align}
    & (1+\gamma^2 \frac{1}{n}) (\sum_{i\in \mathscr{T}}\sum_{j \in \mathscr{C}} (y^t_i - y^c_j)a_{i,j})^2 - \gamma^2 \sum_{i\in \mathscr{T}} \sum_{j \in \mathscr{C}} [(y^t_i - y^c_j) a_{i,j}]^2 \geq 0 \label{feas:quad_max}\\
   & \sum_{i \in \mathscr{T}} a_{ij} \leq 1 \quad \forall j \\
   &\sum_{j \in \mathscr{C}} a_{ij} \leq 1 \quad \forall i \\
   & \sum_{i \in \mathscr{T}} \sum_{j \in \mathscr{C}} a_{ij} = n \\
   & \textrm{Additional constraints to validate the inequality \eqref{feas:quad_max}} \\
   & a_{ij} \in \left \{ 0,1 \right \} \quad \forall i,j  \label{feas:bin_appen}
\end{align}

In the following, we discuss the possible constrints required to validate the square of the inequality in \eqref{feas:max}. Similar to the minimization cases, we consider $\hat{\sigma} > 0$, therefore, will not influence the cases.

\textbf{Case 1:} For this case, we consider $\gamma \geq 0$ and $\sum_{i\in \mathscr{T}}\sum_{j \in \mathscr{C}} (y^t_i - y^c_j)a_{i,j} \geq 0$. The inequality \eqref{feas:max_root} can be written as the following:
\begin{align}
    & \frac{1}{\sqrt{n}}\sum_{i\in \mathscr{T}}\sum_{j \in \mathscr{C}} (t^t_i - y^c_j)a_{i,j} \geq \gamma \hat{\sigma}
\end{align}
Here, $\gamma \geq 0, \hat{\sigma} >0$ and $\sum_{i\in \mathscr{T}}\sum_{j \in \mathscr{C}} (y^t_i - y^c_j)a_{i,j} \geq 0$. So, both sides of the inequality are non-negative. Taking the square will not flip the inequality. Hence, we will have the following constraints.
\begin{align}
    & (1+\gamma^2 \frac{1}{n}) (\sum_{i\in \mathscr{T}}\sum_{j \in \mathscr{C}} (y^t_i - y^c_j)a_{i,j})^2 - \gamma^2 \sum_{i\in \mathscr{T}} \sum_{j \in \mathscr{C}} [(y^t_i -y^c_j) a_{i,j}]^2 \geq 0 \label{max_case1_nonlin} \\
    & \sum_{i\in \mathscr{T}}\sum_{j \in \mathscr{C}} (y^t_i - y^c_j)a_{i,j} \geq 0 \label{max_case1_lin}
\end{align}

\textbf{Case 2:} For this case, we consider $\gamma \leq 0$ and $\sum_{i\in \mathscr{T}}\sum_{j \in \mathscr{C}} (y^t_i - y^c_j)a_{i,j} \leq 0$. The inequality \eqref{feas:max_root} can be written as the following:
\begin{align}
    & \frac{1}{\sqrt{n}}\sum_{i\in \mathscr{T}}\sum_{j \in \mathscr{C}} (y^t_i - y^c_j)a_{i,j} \geq \gamma \hat{\sigma}
\end{align}
Here, $\gamma \leq 0, \hat{\sigma} >0$ and $\sum_{i\in \mathscr{T}}\sum_{j \in \mathscr{C}} (y^t_i - y^c_j)a_{i,j} \leq 0$. As both sides of the inequality are non-positive, the sign of the quadratic constraint will flip and we will have the following constraints. 
\begin{align}
    & (1+\gamma^2 \frac{1}{n}) (\sum_{i\in \mathscr{T}}\sum_{j \in \mathscr{C}} (y^t_i - y^c_j)a_{i,j})^2 - \gamma^2 \sum_{i\in \mathscr{T}} \sum_{j \in \mathscr{C}} [(y^t_i - y^c_j) a_{i,j}]^2 \leq 0 \label{max_case2_nonlin} \\
    & \sum_{i\in \mathscr{T}}\sum_{j \in \mathscr{C}} (y^t_i - y^c_j)a_{i,j} \leq 0 \label{max_case2_lin}
\end{align}

\textbf{Case 3:} For this case, we consider $\gamma \geq 0$ and $\sum_{i\in \mathscr{T}}\sum_{j \in \mathscr{C}} (y^t_i - y^c_j)a_{i,j} \leq 0$. The inequality \eqref{feas:max_root} can be written as the following:
\begin{align}
    & \frac{1}{\sqrt{n}}\sum_{i\in \mathscr{T}}\sum_{j \in \mathscr{C}} (y^t_i - y^c_j)a_{i,j} \geq \gamma \hat{\sigma}
\end{align}
As the left side of the inequality is non-positive and right side is non-negative, the above equation only holds at equality: $\sum_{i\in \mathscr{T}}\sum_{j \in \mathscr{C}} (y^t_i - y^c_j)a_{i,j} = 0$. However, this is already considered in the other cases.

\textbf{Case 4:} For this case, we consider $\gamma \leq 0$ and $\sum_{i\in \mathscr{T}}\sum_{j \in \mathscr{C}} (y^t_i - y^c_j)a_{i,j} \geq 0$. The inequality \eqref{feas:max_root} can be written as the following: 
\begin{align}
    & \frac{1}{\sqrt{n}}\sum_{i\in \mathscr{T}}\sum_{j \in \mathscr{C}} (y^t_i - y^c_j)a_{i,j} \geq \gamma \hat{\sigma}
\end{align}
As $\sum_{i\in \mathscr{T}}\sum_{j \in \mathscr{C}} (y^t_i - y^c_j)a_{i,j} \geq 0$, left side of the inequality is non-negative but, $\gamma \leq 0$ and $\hat{\sigma} >0$ which makes the right side non-positive. Science, this is true for all $\gamma \in \mathbb{R}$ and $\gamma \leq 0$, then, we have,
\begin{align}
    & \frac{1}{\sqrt{n}}\sum_{i\in \mathscr{T}}\sum_{j \in \mathscr{C}} (y^t_i - y^c_j)a_{i,j} \geq \max_{\gamma \in \mathbb{R}} (\gamma \hat{\sigma}) = 0
\end{align}
Therefore, we will have the following constraints. For this case, our non-linear feasibility problem becomes a linear feasibility problem. 
\begin{align}
    & \sum_{i\in \mathscr{T}}\sum_{j \in \mathscr{C}} (y^t_i - y^c_j)a_{i,j} \geq 0
\end{align}

Please note that, the case 4 of the maximization problem is similar to the case 3 of the minimization problem. Therefore, to maintain consistency in the discussion, we will refer to the case 4 of the maximization as the case 3. All cases and resulting constraints are summarized in the following table.  
\begin{table}[htbp]
  \begin{center}
  \caption{Cases and resulting constraints for the maximization problem feasibility formulation of robust Z-test.}
  {\small
  \adjustbox{max width=\textwidth}{
    \begin{tabular}{|c|c|c|}
    \hline
       Case & Case constraints   & Quadratic Constraint \\
    \hline
      1 & $\gamma \geq 0$, $\sum_{i\in \mathscr{T}}\sum_{j \in \mathscr{C}} (y^t_i - y^c_j)a_{i,j} \geq 0$  &  $(1+\gamma^2 \frac{1}{n}) (\sum_{i\in \mathscr{T}}\sum_{j \in \mathscr{C}} (y^t_i - y^c_j)a_{i,j})^2 - \gamma^2 \sum_{i\in \mathscr{T}} \sum_{j \in \mathscr{C}} [(y^t_i - y^c_j) a_{i,j}]^2 \geq 0$ \\
    
    \hline
      2 & $\gamma \leq 0$, $\sum_{i\in \mathscr{T}}\sum_{j \in \mathscr{C}} (y^t_i - y^c_j)a_{i,j} \leq 0$   & $(1+\gamma^2 \frac{1}{n}) (\sum_{i\in \mathscr{T}}\sum_{j \in \mathscr{C}} (y^t_i - y^c_j)a_{i,j})^2 - \gamma^2 \sum_{i\in \mathscr{T}} \sum_{j \in \mathscr{C}} [(y^t_i - y^c_j) a_{i,j}]^2 \leq 0$ \\
    \hline
      3 & $\gamma \geq 0$, $\sum_{i\in \mathscr{T}}\sum_{j \in \mathscr{C}} (y^t_i - y^c_j)a_{i,j} \leq 0$   & \textcolor{black}{No quadratic constraint (Redundant, only true when both inequalities are zero)}  \\
    \hline
      4 & $\gamma \leq 0$, $\sum_{i\in \mathscr{T}}\sum_{j \in \mathscr{C}} (y^t_i - y^c_j)a_{i,j} \geq 0$   & \textcolor{black}{No quadratic constraint}  \\
    \hline
    \end{tabular}}%
  \label{tab:max_cases}%
    }
  \end{center}
\end{table}%

\newpage

\section{QIP Formulation}{\label{app:3}}
In this section, we discuss the Quadratic Integer Programming (QIP) formulation of feasibility problem.
\subsection*{QIP formulation for minimization problems}
Here, we develop the quadratic integer programs (QIPs) for the possible cases of the minimization problem by leveraging the structure of the feasibility formulation.

\textbf{Case 1:} $\gamma \geq 0 $ and $ \sum_{i\in \mathscr{T}} \sum_{j \in \mathscr{C}} [(y^t_i - y^c_j) a_{i,j}] \geq 0 $

Apart from the assignment constraints \eqref{z:angn1}-\eqref{z:bin}, we have to satisfy the constraints \eqref{min_case1_nonlin} and \eqref{min_case1_lin} to calculate $\gamma^*$ in the range $\gamma \geq 0 $. Simplifying the constraint \eqref{min_case1_nonlin} will result to the following:
\begin{align}
& ( \frac{n\gamma^2 }{n+\gamma^2}  )\sum_{i\in \mathscr{T}} \sum_{j \in \mathscr{C}} \left [ (y^t_i - y^c_j) a_{i,j} \right ]^2 \geq [ \sum_{i\in \mathscr{T}}\sum_{j \in \mathscr{C}} (y^t_i - y^c_j)a_{i,j} ]^2 \label{min_case1_nonlin_simple_appn}
\end{align}
As our objective is to find an assignment $a_{i,j}$ that produces the smallest value of $\gamma$ while satisfying the constraints \eqref{min_case1_nonlin_simple_appn} and \eqref{min_case1_lin}, we can exploit the structure of equation \eqref{min_case1_nonlin_simple_appn}. Note that in equation \eqref{min_case1_nonlin_simple_appn}, for any number of samples ($n$) minimum possible value of $\gamma$ is possible when $\sum_{i\in \mathscr{T}} \sum_{j \in \mathscr{C}} \left [ (y^t_i - y^c_j) a_{i,j} \right ]^2$ is maximum and $\left ( \sum_{i\in \mathscr{T}}\sum_{j \in \mathscr{C}} (y^t_i - y^c_j)a_{i,j} \right )^2$ is minimum. At the same time, the assignment should ensure $ \sum_{i\in \mathscr{T}} \sum_{j \in \mathscr{C}} [(y^t_i - y^c_j) a_{i,j}] \geq 0$. Therefore, we need to solve the following coupled problem simultaneously for a given sample size $n$:
\begin{align}
    & \argmax_{a_{i,j} \in \mathscr{M}} \sum_{i\in \mathscr{T}} \sum_{j \in \mathscr{C}} \left [ (y^t_i - y^c_j) a_{i,j} \right ]^2 \quad  \text{S.t.: } \sum_{i\in \mathscr{T}} \sum_{j \in \mathscr{C}} [(y^t_i - y^c_j) a_{i,j}] \geq 0 \text{ and constraints \eqref{z:angn1}-\eqref{z:bin}} \label{min_case1_coup1}\\
    & \argmin_{a_{i,j} \in \mathscr{M}} [ \sum_{i\in \mathscr{T}}\sum_{j \in \mathscr{C}} (y^t_i - y^c_j)a_{i,j} ]^2 \quad  \text{S.t.: } \sum_{i\in \mathscr{T}} \sum_{j \in \mathscr{C}} [(y^t_i - y^c_j) a_{i,j}] \geq 0 \text{ and constraints \eqref{z:angn1}-\eqref{z:bin}}\label{min_case1_coup2_appn}
\end{align}
We can use the following problem to solve the coupling problem \eqref{min_case1_coup1} and \eqref{min_case1_coup2_appn}:
\begin{align*}
    & \max_{a_{i,j} \in \mathscr{M}} \sum_{i\in \mathscr{T}} \sum_{j \in \mathscr{C}} \left [ (y^t_i - y^c_j) a_{i,j} \right ]^2 - [\sum_{i\in \mathscr{T}}\sum_{j \in \mathscr{C}} (y^t_i - y^c_j)a_{i,j} ]^2 \\
    & \text{subject to:} \sum_{i\in \mathscr{T}} \sum_{j \in \mathscr{C}} [(y^t_i - y^c_j) a_{i,j}] \geq 0 \text{ and constraints \eqref{z:angn1}-\eqref{z:bin}}
\end{align*}

\textbf{Case 2:} $\gamma \leq 0 $ and $ \sum_{i\in \mathscr{T}} \sum_{j \in \mathscr{C}} [(y^t_i - y^c_j) a_{i,j}] \leq 0 $

To get $\gamma^*$ in the range $\gamma \leq 0$, we have to find an assignment of pre-defined number of pairs ($n$) while satisfying the constraints \eqref{min_case2_nonlin} and \eqref{min_case2_lin} along with necessary assignment constraints \eqref{z:angn1}-\eqref{z:bin}. Simplifying the equation \eqref{min_case2_nonlin} will result the following:
\begin{align}
& [ \sum_{i\in \mathscr{T}}\sum_{j \in \mathscr{C}} (y^t_i - y^c_j)a_{i,j} ]^2 \geq (\frac{n\gamma^2 }{n+\gamma^2} )\sum_{i\in \mathscr{T}} \sum_{j \in \mathscr{C}} \left [ (y^t_i - y^c_j) a_{i,j} \right ]^2 \label{min_case2_nonlin_simple}
\end{align}
Using the same mechanism we used for case 1, we get the following coupled optimization problem.
\begin{align}
    & \argmax_{a_{i,j} \in \mathscr{M}} [ \sum_{i\in \mathscr{T}}\sum_{j \in \mathscr{C}} (y^t_i - y^c_j)a_{i,j} ]^2 \quad  \text{S.t.: } \sum_{i\in \mathscr{T}} \sum_{j \in \mathscr{C}} [(y^t_i - y^c_j) a_{i,j}] \leq 0 \text{ and constraints \eqref{z:angn1}-\eqref{z:bin}} \label{min_case2_coup11}\\
    & \argmin_{a_{i,j} \in \mathscr{M}} \sum_{i\in \mathscr{T}} \sum_{j \in \mathscr{C}} \left [ (y^t_i - y^c_j) a_{i,j} \right ]^2 \quad  \text{S.t.: } \sum_{i\in \mathscr{T}} \sum_{j \in \mathscr{C}} [(y^t_i - y^c_j) a_{i,j}] \leq 0 \text{ and constraints \eqref{z:angn1}-\eqref{z:bin}} \label{min_case2_coup22}
\end{align}
Now, for a predefined number of required assignments $n$, we can solve the following QIP to find optimal $\gamma$ for case 2.
\begin{align}
    & \argmax_{a_{i,j} \in \mathscr{M}} [ \sum_{i\in \mathscr{T}}\sum_{j \in \mathscr{C}} (y^t_i - y^c_j)a_{i,j} ]^2 - \sum_{i\in \mathscr{T}} \sum_{j \in \mathscr{C}} \left [ (y^t_i - y^c_j) a_{i,j} \right ]^2 \\
    & \text{subject to: } \sum_{i\in \mathscr{T}} \sum_{j \in \mathscr{C}} [(y^t_i - y^c_j) a_{i,j}] \leq 0 \text{ and constraints \eqref{z:angn1}-\eqref{z:bin}} 
\end{align}

\subsection*{QIP formulation for maximization problems}

\textbf{Case 1:} $\gamma \geq 0 $ and $ \sum_{i\in \mathscr{T}} \sum_{j \in \mathscr{C}} [(y^t_i - y^c_j) a_{i,j}] \geq 0 $

The objective of this case is to find the maximum value of $\gamma$ (i.e.,$\gamma^*$) in the range of $\gamma \geq 0$ for a pre-defined number of pairs while satisfying constraints \eqref{max_case1_nonlin} and \eqref{max_case1_lin} along with necessary assignment constraints. With further simplification, constraint \eqref{max_case1_nonlin} can be simplified as the following:
\begin{align}
& [ \sum_{i\in \mathscr{T}}\sum_{j \in \mathscr{C}} (y^t_i - y^c_j)a_{i,j} ]^2 \geq ( \frac{n\gamma^2 }{n+\gamma^2} )\sum_{i\in \mathscr{T}} \sum_{j \in \mathscr{C}} \left [ (y^t_i - y^c_j) a_{i,j} \right ]^2 \label{max_case1_nonlin_simple}
\end{align}
From simplified constraint \eqref{max_case1_nonlin_simple}, we can achieve $\gamma^*$ when $[\sum_{i\in \mathscr{T}}\sum_{j \in \mathscr{C}} (y^t_i - y^c_j)a_{i,j} ]^2$ is maximum and $\sum_{i\in \mathscr{T}} \sum_{j \in \mathscr{C}} [ (y^t_i - y^c_j) a_{i,j}  ]^2$ is minimum. So, we can solve the case 1 of the maximization problem by solving the following coupled problem simultaneously for a given sample size $n$:
\begin{align}
    & \argmax_{a_{i,j} \in \mathscr{M}} [ \sum_{i\in \mathscr{T}}\sum_{j \in \mathscr{C}} (y^t_i - y^c_j)a_{i,j} ]^2 \quad  \text{S.t.: } \sum_{i\in \mathscr{T}} \sum_{j \in \mathscr{C}} [(y^t_i - y^c_j) a_{i,j}] \geq 0 \text{ and constraints \eqref{z:angn1}-\eqref{z:bin}} \label{max_case1_coup1}\\
    & \argmin_{a_{i,j} \in \mathscr{M}} \sum_{i\in \mathscr{T}} \sum_{j \in \mathscr{C}} \left [ (y^t_i - y^c_j) a_{i,j} \right ]^2 \quad  \text{S.t.: } \sum_{i\in \mathscr{T}} \sum_{j \in \mathscr{C}} [(y^t_i - y^c_j) a_{i,j}] \geq 0 \text{ and constraints \eqref{z:angn1}-\eqref{z:bin}} \label{max_case1_coup2}
\end{align}
We can formulate the following QIP by combining the coupling problem \eqref{max_case1_coup1} and \eqref{max_case1_coup2}.
\begin{align}
    & \argmax_{a_{i,j} \in \mathscr{M}} [ \sum_{i\in \mathscr{T}}\sum_{j \in \mathscr{C}} (y^t_i - y^c_j)a_{i,j} ]^2 - \sum_{i\in \mathscr{T}} \sum_{j \in \mathscr{C}} \left [ (y^t_i - y^c_j) a_{i,j} \right ]^2 \\
    & \text{subject to: } \sum_{i\in \mathscr{T}} \sum_{j \in \mathscr{C}} [(y^t_i - y^c_j) a_{i,j}] \geq 0 \text{ and constraints \eqref{z:angn1}-\eqref{z:bin}} 
\end{align}

\textbf{Case 2:} $\gamma \leq 0 $ and $ \sum_{i\in \mathscr{T}} \sum_{j \in \mathscr{C}} [(y^t_i - y^c_j) a_{i,j}] \leq 0 $

To get $\gamma^*$ in the range $\gamma \leq 0$, we have to find an assignment of pre-defined number of pairs ($n$) while satisfying the constraints \eqref{max_case2_nonlin} and \eqref{max_case2_lin} along with necessary assignment constraints. Simplifying the equation \eqref{max_case2_nonlin} will result the following:
\begin{align}
& ( \frac{n\gamma^2 }{n+\gamma^2}  )\sum_{i\in \mathscr{T}} \sum_{j \in \mathscr{C}} \left [ (y^t_i - y^c_j) a_{i,j} \right ]^2 \geq [\sum_{i\in \mathscr{T}}\sum_{j \in \mathscr{C}} (y^t_i - y^c_j)a_{i,j} ]^2 \label{max_case2_nonlin_simple}
\end{align}
Using the same mechanism we used for case 1, we get the following coupled optimization problem.
\begin{align}
    & \argmax_{a_{i,j} \in \mathscr{M}} \sum_{i\in \mathscr{T}} \sum_{j \in \mathscr{C}} \left [ (y^t_i - y^c_j) a_{i,j} \right ]^2 \quad  \text{S.t.: } \sum_{i\in \mathscr{T}} \sum_{j \in \mathscr{C}} [(y^t_i - y^c_j) a_{i,j}] \leq 0 \text{ and constraints \eqref{z:angn1}-\eqref{z:bin}} \label{min_case2_coup1}\\
    & \argmin_{a_{i,j} \in \mathscr{M}} [ \sum_{i\in \mathscr{T}}\sum_{j \in \mathscr{C}} (y^t_i - y^c_j)a_{i,j} ]^2 \quad  \text{S.t.: } \sum_{i\in \mathscr{T}} \sum_{j \in \mathscr{C}} [(y^t_i - y^c_j) a_{i,j}] \leq 0 \text{ and constraints \eqref{z:angn1}-\eqref{z:bin}} \label{min_case2_coup2}
\end{align}
The problems in \eqref{min_case2_coup1} and \eqref{min_case2_coup2} can be combined in the following problem:
\begin{align}
    & \argmax_{a_{i,j} \in \mathscr{M}} \sum_{i\in \mathscr{T}} \sum_{j \in \mathscr{C}} [ (y^t_i - y^c_j) a_{i,j}  ]^2 - [ \sum_{i\in \mathscr{T}}\sum_{j \in \mathscr{C}} (y^t_i - y^c_j)a_{i,j} ]^2  \\
    & \text{subject to:} \sum_{i\in \mathscr{T}} \sum_{j \in \mathscr{C}} [(y^t_i - y^c_j) a_{i,j}] \leq 0  \text{ and constraints \eqref{z:angn1}-\eqref{z:bin}}
\end{align}

\end{document}